\documentclass[10pt]{amsart}

\usepackage{amssymb,amsmath}
\usepackage{amsthm}
\usepackage{esint}
\usepackage[dvipsnames]{xcolor}
\usepackage{comment}
\usepackage{mathrsfs}
\usepackage{graphicx}
\usepackage{tikz}
\usepackage{float}
\usepackage{hyperref}

\title{Aggregation-diffusion phenomena: from microscopic models to free boundary problems}

\author[I. Kim]{Inwon Kim}
\address{Department of Mathematics, UCLA,  Los Angeles, CA} 
\email{ikim@math.ucla.edu}

\author[A. Mellet]{Antoine Mellet}
\address{Department of Mathematics, University of  Maryland, College Park, MD}
\email{mellet@umd.edu}

\author[J. S.-H. Wu]{Jeremy Sheung-Him Wu}
\address{Department of Mathematics, UCLA,  Los Angeles, CA}
\email{jeremywu@math.ucla.edu}

\thanks{I. Kim was partially supported by NSF Grant DMS-2153254. \\
A. Mellet was partially supported by NSF Grant DMS-2009236 and DMS-2307342.
}

\def\R{\mathbb R}
\def\eps{\varepsilon}
\def\e{\varepsilon}
\def\P{\mathcal P}

\def\vphi{\varphi}
\def\pa{\partial}
\def\na{\nabla}
\def\div{\mathrm{div}\,}
\def\supp{\mathrm{Supp}\,}
\def\BV{\mathrm{BV}}
\def\ds{\displaystyle}
\newcommand{\E}{\mathscr{E}}
\newcommand{\J}{\mathscr{J}}
\newcommand{\G}{\mathscr{G}}
\newcommand{\F}{\mathscr{F}}

\def\H{\mathcal H}
\def\P{\mathcal P}

\numberwithin{equation}{section}

\newtheorem{theorem}{Theorem}[section]
\newtheorem{theorem*}{Theorem}

\newtheorem{lemma}[theorem]{Lemma}

\newtheorem{proposition}[theorem]{Proposition}

\begin{document}

\begin{abstract}
This paper reviews (and expands) some recent results on the modeling of aggregation-diffusion phenomena at various scales, focusing on the emergence of collective dynamics as a result of the competition between attractive and repulsive phenomena -  especially (but not exclusively) in the context of attractive chemotaxis phenomena.

At microscopic scales, particles (or other agents) are represented by spheres of radius $\delta>0$ and we discuss both soft-sphere models (with a pressure term penalizing the overlap of the particles) and hard-sphere models (in which overlap is prohibited). The first case leads to so-called ``blob models" which have received some attention recently as a tool to approximate non-linear diffusion by particle systems. The hard-sphere model 
is similar to a classical model for congested crowd motion. 
We will review well-posedness results for these models and discuss their relationship to classical continuum description of aggregation-diffusion phenomena in the limit $\delta\to0$: the classical nonlinear drift diffusion equation and its incompressible counterpart.

In the second part of the paper, we discuss  recent results on the emergence and evolution of sharp interfaces when a large population of particles is considered at appropriate space and time scales:
At some intermediate time scale, phase separation occurs and a sharp interface appears which evolves according to a Stefan free boundary problem  (and the density function eventually relaxes to a characteristic function - metastable steady state for the original problem). At a larger time scale the attractive forces lead to surface tension phenomena and the evolution of the sharp interface can be described by a Hele-Shaw free boundary problem with surface tension. At that same time scale, we will also discuss the emergence of contact angle conditions for problems set in bounded domains. 

\end{abstract}

\maketitle

\setcounter{tocdepth}{1}

\tableofcontents

\section{Introduction}
\subsection{Aggregation-diffusion phenomena}
The study of emergent collective dynamics in large systems of interacting agents
has applications in many areas of applied sciences and has been  the source of many interesting mathematical problems for decades.
Classical examples in mathematical biology include the 
collective dynamics of micro-organisms such as cells or amoeba via chemotaxis \cite{Keller_Segel,Patlak}
or that of larger living organisms (insect swarms, the motion of animal herds or human crowds) \cite{TBL,BT16,MRSV}.
Throughout this paper, we will refer to these organisms as active particles
and assume that they are experiencing two competing effects: repulsion, modeled by nonlinear diffusion (or even density constraint)  and attraction, modeled by nonlocal interactions. 
The simple  mathematical model at the center of this paper is the following aggregation-diffusion equation
\begin{equation}\label{eq:AD}
 \pa_t \rho + \div(\rho \na (G* \rho ) ) = \div(\rho \na f'(\rho)) \qquad \mbox{ in } [0,\infty)\times\R^d, \quad \rho(t=0) = \rho_{in},
\end{equation}
which describes the evolution of the density of particles $\rho(t,x)$ and 
provides a continuum description of the particle systems at some macroscopic scale (see \cite{CCY19} and reference therein). 

The non-linearity $\rho\mapsto f(\rho)$ describes the repulsive force which models a natural volume exclusion principle (or incompressibility of the particles) and will play a central role throughout this paper. 
The standard linear diffusion corresponds to $f(\rho)=\rho\log\rho$, but  we will consider pressure laws which penalize values of the density above a critical threshold $\rho^*$ and lead to degenerate diffusion for small values of $\rho$. Taking  $\rho^*=1$ (to simplify the notations), we will focus on the following power-law:
\begin{equation}\label{eq:fm} 
f_m(\rho): = \frac{\rho^m}{m-1}  , \qquad \forall \rho\geq0, \qquad \qquad m>1,
\end{equation}
and its limit $m\to\infty$ (which enforces the congestion constraint $\rho\leq 1$):
\begin{equation}\label{eq:finfty} 
f_\infty(\rho): = \begin{cases} 0 & \mbox{ when } \rho\in [0,1],  \\ \infty & \mbox{ when $\rho>1$}\end{cases}
\end{equation}
(in that case \eqref{eq:AD} should be written as \eqref{eq:LHS}).
More general convex functions can be considered as well.

The interaction kernel $G$ in \eqref{eq:AD} describes long range interactions. 
We will focus our analysis on kernels that are radially symmetric, {\bf non-negative} and {\bf integrable}:
$$
G=G(|x|) \geq 0, \qquad G\in L^1(\R^d).
$$
In many applications, the interaction kernels will be radially decreasing\footnote{Radially decreasing kernels model an attractive force between particles.} at least for large $|x|$ though we do not need such an assumption. 
One can think of the potential $\phi=G*\rho$ as a chemical sensing function and the drift term in \eqref{eq:AD} models the motion of the agents toward higher values of that function.
An important example of an attractive kernel is  $G$ given by the equation
\begin{equation}\label{eq:G}
\sigma G - \eta \Delta G =  \delta_{x=0} \qquad \mbox{ in } \R^d, \qquad \lim_{|x|\to\infty} G(x) =0,
\end{equation}
which satisfies $G(x)\sim |x|^{2-d}$ when $|x|\to 0$ and $\int_{\R^d} G(x)\, dx =\sigma^{-1}$.
When $f$ is given by \eqref{eq:fm} and $G$ by \eqref{eq:G}, Equation
 \eqref{eq:AD} is the classical (parabolic-elliptic)  Patlak-Keller-Segel model (PKS) for chemotaxis \cite{Keller_Segel,Patlak} which describes the collective motion of cells that are attracted by a self-emitted chemical substance (we refer to Hillen and Painter~\cite{HP} for a review on various modeling aspects): 
\begin{equation}\label{eq:KellerSegel}
\begin{cases}
\pa_t \rho + \div(\rho \na \phi ) = \Delta \rho^m, \\
\sigma \phi - \eta \Delta \phi =\rho.
\end{cases}
\end{equation}
While we will not exclusively focus on this model, all the results presented here will hold in particular for \eqref{eq:KellerSegel} when $m>2$. 
 
 \medskip

\noindent{\bf Purpose of this paper:}
Aggregation-diffusion equations such as \eqref{eq:AD} have been an active area of mathematical research for the past few decades. 
Well-posedness, blow-ups, long time behavior, steady states and energy minimizers have all been studied intensely and we will not attempt to present a complete picture of the existing theory
(see \cite{BR,BRB,CC06,CHVY,CCY19,DXY,K05,Sugiyama} and references therein).
Instead, our focus will be on two recent trends in the analysis of these phenomena at both smaller and larger scales than the continuum model \eqref{eq:AD}.
 
In the first part of the paper, we will describe a microscopic model for aggregation  phenomena that includes a repulsive term taking into account the finite size of the particles, modeled by balls of radius $\delta>0$. When $f$ is given by \eqref{eq:fm}, this model is similar to the blob-method introduced for approximating  nonlinear diffusion equations \cite{BE23,CCP19,carrillo2023nonlocal,CB16,CEHT23,LM01}. 
Since this model allows for the overlap of the balls, we will refer to it as a soft-sphere model.
We will also consider a hard-sphere microscopic model, which includes a strict non-overlapping condition, and which has been used to model congested crowd motion \cite{FM,KMW,M18,MRS,MRSV,Maury11}. 
At the macroscopic scale, this hard-sphere model formally corresponds to \eqref{eq:AD} with $f$ given by \eqref{eq:finfty} and is related to Hele-Shaw free boundary problems.
The purpose of this first part of the paper  is thus  to review these microscopic models (their derivation and their properties) and to discuss the connection with the continuum model \eqref{eq:AD}.

In the second part of this paper, we will consider these same phenomena at a larger scale:  When $m\in (2,\infty]$, we will show that equation \eqref{eq:KellerSegel}  leads  to phase separation (or phase segregation) and to the emergence of sharp interfaces.
Phase separation, or the ability of a given population to organize itself into segregated groups with sharp edges, is an important feature of biological aggregation. Examples include animals moving as a homogeneous patch~\cite{PE,LDR} (e.g. insect swarms), pattern formation in cell tissues~\cite{VS,CMSTT}, or 
the formation of membraneless organelles in eukaryotic cells\footnote{Membraneless organelles are mesoscopic structures found in the nucleus and cytoplasm of cells, characterized by higher protein density and weaker molecular motion than the surrounding medium, but not separated from it by a membrane. They play a crucial role in the cell's operations by allowing for increased rates of biochemical reactions.} \cite{B09,HN,MK}. 
The fact that attractive nonlocal interactions can explain this phenomena has been previously pointed out \cite{TBL,FBC} and we present here a rigorous approach based on the recent work \cite{KMW2,M23} to derive these phenomena from \eqref{eq:AD} in regimes corresponding to a very large population of particles observed from far away over large time scales (or, equivalently, small attraction range).
Characterizing the dynamics of these interfaces, and thus describing the collective behavior of the particles after aggregation, is a fascinating and challenging problem and we will present some recent results in this direction:
At some appropriate time scale, the evolution of this interface is described by free boundary problems (Stefan and Hele-Shaw), which include surface tension effects as well as contact angle conditions and which are reminiscent of classical models for the evolution of interfaces in fluid dynamics.

\subsection{Finite time blow-up VS global well-posedness}
Before presenting these various models, we recall that 
since the attractive potential $G$ leads to concentration, the well-posedness of \eqref{eq:AD} globally in time is not obvious. 
Preventing blow-ups (the formation of Dirac masses) in the mathematical model is relatively natural from a modeling point of view since one would expect that the particles (if they represent living organisms) have some built-in incompressibility which prevents overlap and thus limits the maximum density.
For the classical PKS model \eqref{eq:KellerSegel}, it is known that the standard {\bf linear diffusion} ($m=1$)  is not enough to prevent finite time blow-up, but that for $m>m_c := 2-\frac{2}{d}$  solutions remain bounded in $L^\infty$ uniformly in time (see \cite{CC06,Sugiyama,BRB,BR}) and global well-posedness holds for this equation. 
In the second half of the paper we assume that $m>2$. Under this stronger condition, we show that solutions of   \eqref{eq:KellerSegel} are not only bounded, but experience phase separation phenomena.

\medskip

\subsection{Micro, macro and geometric models}
We now present the various  models that will be discussed in this paper. 
In order to keep this introduction relatively short, we did not include references below. 
Relevant references for each of these models will be discussed in the corresponding sections of the paper.
Figure \ref{fig:intro} (at the end of this section) summarizes the relationship between the various models.
\medskip

\noindent\underline{Microscopic ``agent based'' models (soft-sphere and hard-sphere)}
At the microscopic level with $N\gg1$ particles, each particle is represented by a ball $B_\delta(x_i(t))$  with radius $\delta>0$, whose center $x_i(t)$ evolves according to a first order ODE for $i=1,\dots,N$. In the soft-sphere model, the overlap of each ball $B_\delta(x_i(t))$ is allowed, but penalized, resulting in a repulsive force depending on the local density. This leads to the following microscopic model (see Section \ref{sec:blob}):
$$
\dot{x}_i (t) 
 =  \frac 1 N \sum_{j=1}^N \na \widetilde G_\delta    (x_j(t)-x_i(t)) - \na\int K_\delta (x_i(t) -y) f'\left(\frac 1 N \sum_{j=1}^N  K_\delta (y-x^j(t))\right) \, dy, \quad \forall i=1,\dots,N,
$$
where $K_\delta=\frac{1}{|B_\delta|} \chi_{B_\delta(0)}$ (or another approximation of unity supported in $B_\delta(0)$) and  $\widetilde G_\delta    = K_\delta* G*K_\delta$. The corresponding empirical distribution 
$\rho(t,x) = \frac 1 N \sum_{i=1}^N \delta(x-x_i(t))$ solves the non-local PDE
\begin{equation}\label{eq:NLSS}
\pa_t \rho + \div(\rho \na (\widetilde G_\delta   * \rho ) ) = \div( \rho  \na K_\delta*f'( K_\delta*\rho)) \qquad \mbox{ in } [0,\infty)\times\R^d
\end{equation}
with initial condition
\begin{equation}\label{eq:init}
\rho(0,x) = \frac 1 N \sum_{i=1}^N \delta(x-x_i(0)).
\end{equation}
Similar equations and corresponding systems of ODEs have been introduced and used in particular to develop numerical methods for nonlinear diffusion models.
In Section \ref{sec:blob}, we will recall several results concerning the well-posedness of  \eqref{eq:NLSS} for general initial data (not necessarily empirical distributions) and review convergence results as $N\to\infty$ (mean-field limit - see Theorem \ref{thm:stab}) and $\delta\to0$ (Theorem~\ref{thm:convdelta}).
\medskip

The microscopic hard-sphere model is a similar model in which 
 overlap is not permitted: The constraint 
 $$| x_i(t)-x_j(t)|\geq 2\delta, \qquad \forall i,j=1,\dots,N, \qquad i\neq j,$$ 
is strictly enforced, leading to a well-posed system of ODE with constraints (see Section \ref{sec:micro}). 
A corresponding PDE can be obtained formally by taking the limit $m\to\infty$ in \eqref{eq:NLSS}, leading to the following equation (see Section \ref{sec:HS}):
 \begin{equation}\label{eq:NLHS}
\pa_t \rho + \div(\rho \na (\widetilde G_\delta   * \rho ) ) = \div( \rho  \na K_\delta* p), \quad p\in \pa f_\infty(K_\delta*\rho)\qquad \mbox{ in } [0,\infty)\times\R^d,
\end{equation}
which includes a pressure function $p(t,x)$ playing the role of a Lagrange multiplier for the non-local constraint $K_\delta*\rho \leq 1$. The condition $p\in \pa f_\infty(K_\delta*\rho)$\footnote{$\pa f_\infty$ denotes the subdifferential of the convex function $f_\infty$ and is sometimes called the {\it Hele-Shaw graph}. It is a multi-valued function given by
$$
 \pa f_\infty(\rho):=\begin{cases} 0 & \mbox{ if } \rho<1 \\ [0,\infty) & \mbox{ if } \rho=1 \\ \infty & \mbox{ if } \rho>1\end{cases}
$$
}  is a shorthand notation for the three conditions
$$
K_\delta*\rho (t,x)\leq 1, \qquad p(t,x)\geq 0\, \qquad p(t,x)(1-K_\delta*\rho(t,x))=0 \qquad \mbox{ a.e.}
$$
(note that \eqref{eq:NLSS} can be written in a similar form with $p=\pa f(K_\delta*\rho)$). However, while \eqref{eq:NLHS} is indeed related to a well-posed microscopic model when the initial condition is an empirical distribution \eqref{eq:init}, we will see that its well-posedness for general initial conditions is far from clear (see Section \ref{sec:mm}).

\medskip

\noindent\underline{Macroscopic models (soft-sphere and hard-sphere).}
In the limit $\delta\to 0$,  Equation \eqref{eq:NLSS} leads to the classical aggregation-diffusion equation which we recall here:
\begin{equation}\label{eq:LSS}
\pa_t \rho + \div(\rho \na (G* \rho ) ) = \div( \rho  \na  f'( \rho))\qquad \mbox{ in } [0,\infty)\times\R^d,
\end{equation}
while the hard-sphere model \eqref{eq:NLHS} yields:
\begin{equation}\label{eq:LHS}
\pa_t \rho + \div(\rho \na (G * \rho ) ) = \div( \rho  \na   p), \qquad p\in \pa f_\infty(\rho)\qquad \mbox{ in } [0,\infty)\times\R^d .
\end{equation}
A similar  equation was introduced as a model for congested crowd motion and some variations of it were derived in many frameworks, in particular as mechanical models for tumor growth. 
Interpretations and properties of this model are recalled in Sections \ref{sec:hs}-\ref{sec:proj}.
It is closely related to the classical Hele-Shaw free boundary problem with active potential (see Section \ref{sec:hsFBP}). 

\medskip

\noindent \underline{Geometric models: Phase separation and free boundary problems.}
The second half of this paper is devoted to phase separation, sharp interface limits and the derivation of free boundary problems from  \eqref{eq:LSS}  when $f$ is given by \eqref{eq:fm} with $m>2$ or from \eqref{eq:LHS}.
The {\it sharp interface limit} corresponds to the regime in which the size of the support of $\rho$ is very large compared to the interfacial region, and we will derive mathematical models for the evolution of this interface.
The results of Sections \ref{sec:FBP} - \ref{sec:ca} can be summarized as follows:
When the total mass $\int \rho\, dx =: \eps^{-d}$ is very large ($\eps\ll1$), we observe the apparition of an interface separating regions of low and high densities, which happens at time scale of order $\eps^{-2}$. After that time, the evolution of the interface can be described by a {\it Stefan free boundary problem}. In that regime, the motion of the interface is driven by the evolution of the density in the bulk. This bulk density eventually relaxes toward a constant (stable) value $\theta$. 
At that point, the motion of the interface stops, at least at this time scale. But such a state is only metastable and the slower evolution of the interface, at time scale $\eps^{-3}$, can be described by a one-phase Hele-Shaw free boundary problem with surface tension.
\medskip

To justify this analysis, we consider a large number  of particles, $\int \rho_{in}(x)\, dx = \eps^{-d} \gg 1$.
Rescaling the variables $x\mapsto \eps x$ and $t\mapsto \eps^2 t$ leads to the equation
\begin{equation}\label{eq:LSSeps}
\pa_t \rho + \div(\rho \na (G_\eps* \rho ) ) = \div( \rho  \na  f'( \rho)) \qquad \mbox{ in } [0,\infty)\times\R^d
\end{equation}
with $ G_\eps(x) := \eps^{-d} G(\eps^{-1} x)$ and $\int \rho_{in}(x)\, dx=1$ (we can also consider the corresponding hard-sphere model).
Recalling that $G\in L^1(\R^d)$ and denoting 
$$\int_{\R^d} G(x)\, dx = \frac{1}{\sigma}<\infty$$ 
(which is satisfied for \eqref{eq:G}), 
we rewrite \eqref{eq:LSSeps} as follows:
\begin{equation}\label{eq:CHeps}
\pa_t \rho + \div\left(\rho \na \int G_\eps(x-y) (\rho(y)-\rho(x))\, dy\right) = \div( \rho  \na  h'( \rho)) \qquad \mbox{ in } [0,\infty)\times\R^d,
\end{equation}
where 
\begin{equation}\label{eq:hh}
h(\rho):=f( \rho)-\frac{1}{2\sigma} \rho^2 + a \rho
\end{equation}
for some constant $a\in \R$ - the addition of the term $a\rho$ in the function $h$ does not affect the equation but will be important for the energy. When $f$ is given by \eqref{eq:fm} with $m>2$, 
the function $h''(\rho)=f''(\rho)-\frac{1}{\sigma}$ is negative  for small $\rho$ and positive for large $\rho$. 
This is a fundamental property which implies that the limiting equation $\pa_t \rho  = \div( \rho  \na  h'( \rho))$ is an ill-posed forward-backward diffusion equation and leads to phase separation.
In fact, the function $h$ is a {\bf double well-potential}:  
When $f$ is given by \eqref{eq:fm} with $m>2$, the function $h$  in \eqref{eq:hh} becomes
\begin{equation}\label{eq:hm}
 h_m(\rho) = \frac{\rho}{m-1} \left( \rho^{m-1} - (m-1) \theta_m^{m-2}\rho + (m-2)\theta_m^{m-1}\right), \qquad 
 \theta_m = \left(\frac{1}{2\sigma}\right)^\frac{1}{m-2}
 \end{equation}
where we took $a = \frac{m-2}{m-1} \theta_m^{m-1}$. This function is a double-well potential with a (singular) well at $\rho=0$ and a (smooth) well at $\rho=\theta_m$ (see Figure \ref{fig:1}).
\begin{figure}
 	 	 \includegraphics[width=.25\textwidth]{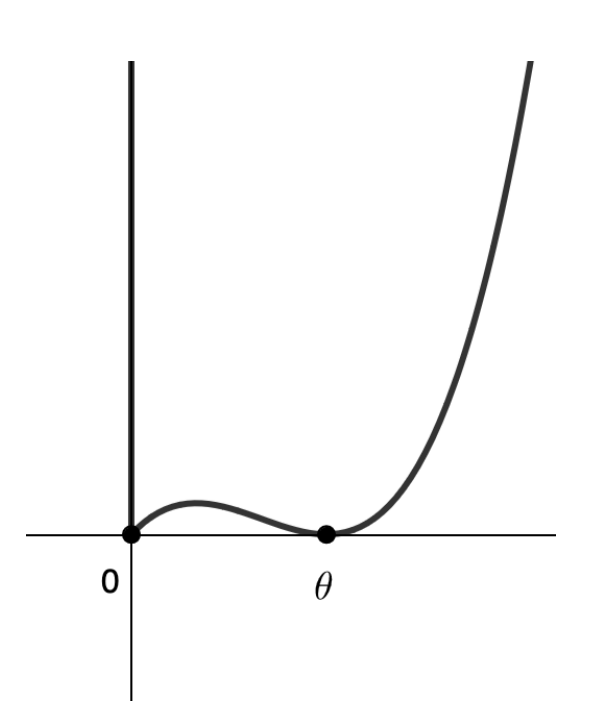}	\hspace*{30pt} \includegraphics[width=.25\textwidth]{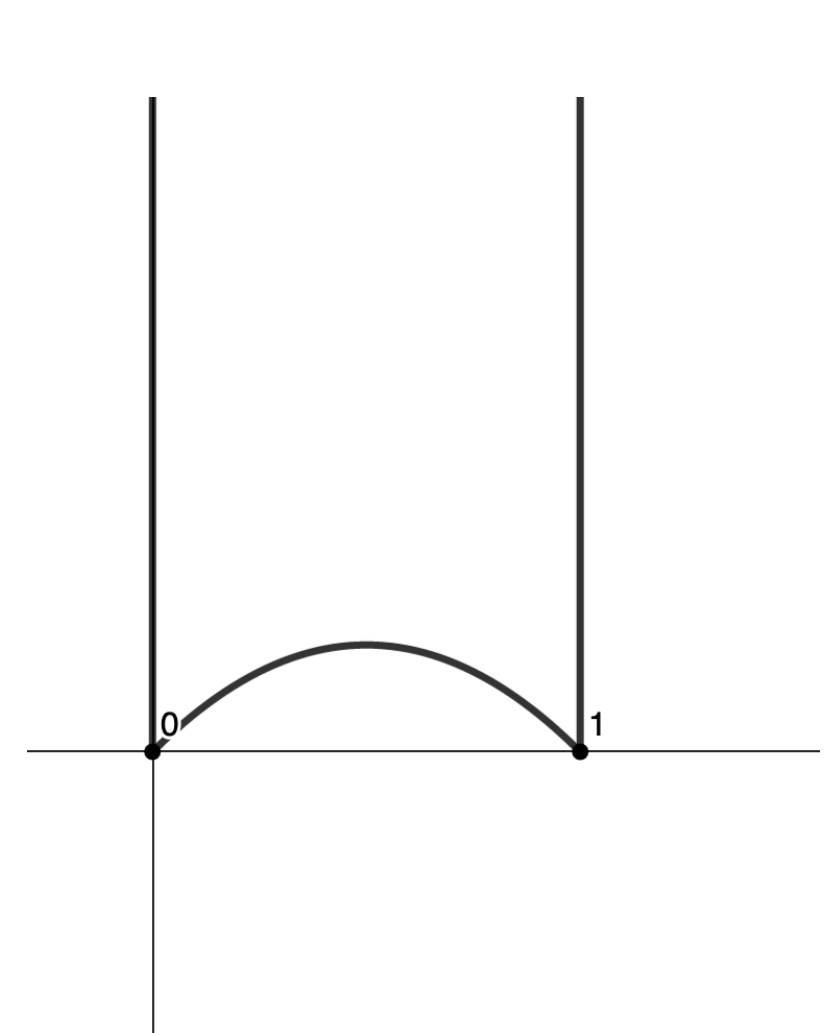}
			\vspace*{-10pt}			
 	 		\caption{The double-well potential   $h_m(\rho) $ when  $m=3$ (left) and $h_\infty(\rho)$ (right)}
 	 		\label{fig:1}
 	 	\end{figure}
When $f=f_\infty$, we take $a=\frac 1{2\sigma }$ in \eqref{eq:hh}, and find
\begin{equation}\label{eq:hinfty}
h_\infty(\rho) := 
\begin{cases}
\frac  1{2\sigma}\rho(1-\rho) & \mbox{ if } 0\leq \rho\leq 1;\\
\infty & \mbox{  otherwise}
\end{cases}
\end{equation}
which  is a double-well (or double obstacle) potential, with wells at $\rho=0$ and $\rho=1$ (see Figure \ref{fig:1}).

\medskip

We do not need to restrict ourselves to $f$ given by  \eqref{eq:fm} or \eqref{eq:finfty} for the results below. All we need is that the function $\rho\mapsto h(\rho)$ is a double-well potential with stable region $\{0\}\cup (\theta,\infty)$ for some $\theta>0$.
Under this condition,
\eqref{eq:CHeps} is a nonlocal approximation of a Cahn-Hilliard equation whose  sharp interface limit $\eps\to0$ is a classical problem. 
We will show (in terms of $\Gamma$-convergence of the energy - see Theorem \ref{thm:stefan})
 that the limit $\eps\to0$ of~\eqref{eq:LSSeps} or~\eqref{eq:CHeps} leads to a generalized
Stefan problem:
\begin{equation}\label{eq:stefan0}
\pa_t \rho = \div(\rho\na {h^{**}}'(\rho))
\end{equation}
where $h^{**}$ denotes the convex hull of $h$.
When the initial density does not take values in the unstable phase, that is if
$|\{\rho_{in}(x)\in (0,\theta)\}|=0$, then  \eqref{eq:stefan0} is a weak formulation of
\begin{equation}\label{eq:stef}
\begin{cases}
\rho \geq \theta, \quad \pa_t \rho = \div(\rho\na {h^{**}}'(\rho)) \qquad \mbox{ in } E(t) \\
\rho = 0\qquad \mbox{ in } \R^d\setminus E(t) \\
V = - \na  {h^{**}}'(\rho) \cdot n.
\end{cases}
\end{equation}
where $V$ denotes the normal velocity of $\pa E(t)$.
The evolution of $\rho$ in the set $E(t)$ (described by a nonlinear diffusion equation) is thus responsible for the motion of the interface $\pa E(t)$ and will eventually lead to constant density, that is $\rho(t)\to \theta \chi_{E}$  for some (constant) set $E$.

\medskip

If $\rho_{in}(x) = \theta\chi_{E_{in}}$, then the solution of \eqref{eq:stefan0} is constant in time which indicates that the corresponding solution of \eqref{eq:LSSeps} evolves very slowly when $\eps\ll1$ (characteristic functions are metastable for \eqref{eq:LSSeps}). In order to characterize the evolution of such solutions over larger time scale (of order $\eps^{-3}$ compared to the microscopic time scale), we  consider (in Section \ref{sec:st}) the equation
\begin{equation}\label{eq:LSSeps2}
\eps \pa_t \rho^\eps + \div(\rho^\eps \na (G_\eps* \rho^\eps ) ) = \div( \rho^\eps  \na  f'( \rho^\eps))\qquad \mbox{ in } [0,\infty)\times\R^d
\end{equation}
with initial data $\rho_{in}(x) = \theta\chi_{E_{in}}$ for some set $E_{in}$ with finite perimeter.
When $G$ is the Newtonian kernel \eqref{eq:G}, 
the evolution of $\rho^\eps$, solution of \eqref{eq:LSSeps2}, is then asymptotically described by the Hele-Shaw free boundary problem with surface tension (see Theorem \ref{thm:HS}):
\begin{equation}
\label{eq:HS}
\begin{cases}
\Delta p = 0 & \mbox{ in } E(t), \\
p = \gamma  \kappa & \mbox{ on } \pa E(t),  \\
V = -\na p \cdot \nu & \mbox{ on } \pa E(t) .
\end{cases}
\end{equation}
where $\kappa(t,x)$ denotes the mean-curvature of $\pa E(t)$ (with the convention that $\kappa\geq 0$ when $E$ is convex) and $V$ denotes the normal velocity of the interface $\pa E(t)$.
The constant $\gamma$ depends on $m$ and $\sigma$ (see \eqref{eq:sigma}).

\medskip

This shows that while the Stefan problem \eqref{eq:stefan0} describes the phase separation phenomena that takes place for the solutions of \eqref{eq:LSSeps} when $\eps\ll1$, the resulting clusters (regions when $\rho = \theta$) will continue to evolve over a much larger time scale. This evolution, due to to the attractive  interactions between the particles, is akin to the evolution of an interface separating two immiscible fluids under the effect of surface tension phenomenon.
This is not completely surprising: Surface tension phenomena can be explained by stating that the bond between two molecules of the same fluid  (cohesion) is stronger than that between two molecules of different fluids (adhesion). The same can be said about our particles, which are happier when surrounded by other particles than by vacuum.
\medskip

Another important phenomenon that often comes together with surface tension is the notion of {\bf contact angle}, which appears when the free surface come in contact with a fixed boundary. In most of the paper we will look at problems set in the whole space $\R^d$, but in Section \ref{sec:ca} we look at problems set in a bounded domain and show that   contact angle conditions are a natural byproduct of the sharp interface limit discussed above.
More precisely, we will consider the initial boundary value problem
\begin{equation}\label{eq:LSSbd}
\begin{cases}
\eps \pa_t \rho + \div(\rho \na \phi_\eps  ) = \div( \rho  \na  f'( \rho)) & \mbox{ in }[0,\infty) \times  \Omega  \\
\rho \na [ -f'( \rho) + \phi_\eps ]\cdot n  =0 & \mbox{ on }[0,\infty)\times \pa \Omega \\
\rho(x,0)=\rho_{in} & \mbox{ in } \Omega.
\end{cases}
\end{equation}
In the context of chemotaxis, $\phi_\eps $ will be the  solution of an elliptic boundary value problem in $\Omega$  (see \eqref{eq:phieps}). In that case we will derive a contact angle condition which depends on the boundary conditions for $\phi_\eps $ (see \eqref{eq:CAC}).
In some experimental settings, it can also be interesting to keep $\phi_\eps  = G_\eps*\rho$ (where $\rho$ is extended by $0$ to $\R^d$ to make sense of the convolution), modeling interactions that only depends  on the distance between the particles. In  that case, $\Omega$ acts as an obstacle and the limiting free boundary problem is \eqref{eq:HS} supplemented with contact angle condition $\alpha=\pi$ (i.e. the contact is tangential). 
But  general contact angle conditions can be recovered as well by taking into account the interactions of our particles with the fixed boundary (Section \ref{sec:CAB}).

\medskip

\noindent\underline{Longer time scale.}
As a final note, we observe  that stationary solutions for \eqref{eq:LSS} and global minimizers of the corresponding energy have been  the subject of intense research in the past two decades. We will not attempt to give a full review of these results here, but we recall that in many settings it is known that stationary solutions must be radially decreasing.
For example  when $G$ is the Newtonian attractive kernel and $f$ is given by \eqref{eq:fm} with $m>1$, the unique stationary state in two dimensions is radially symmetric and decreasing (see for instance  \cite{CHVY,DXY}).
In contrast, both \eqref{eq:stefan0} and \eqref{eq:HS} have stationary solutions that are not radially decreasing, but consist of clusters separated by vacuum. 
This suggests that when $\eps\ll1$, the solutions of \eqref{eq:LSSeps} and  \eqref{eq:LSSeps2} may first approach some metastable equilibrium, close to some steady states of the limiting free boundary problems, but eventually these clusters will coalesce and converge as $t\to\infty$ to a radially symmetric stationary state (which consists of a unique cluster). Such dynamics are indeed observed numerically (see for instance \cite{CCY19}).

\begin{figure}	\label{fig:intro}
\includegraphics[width=.9\textwidth]{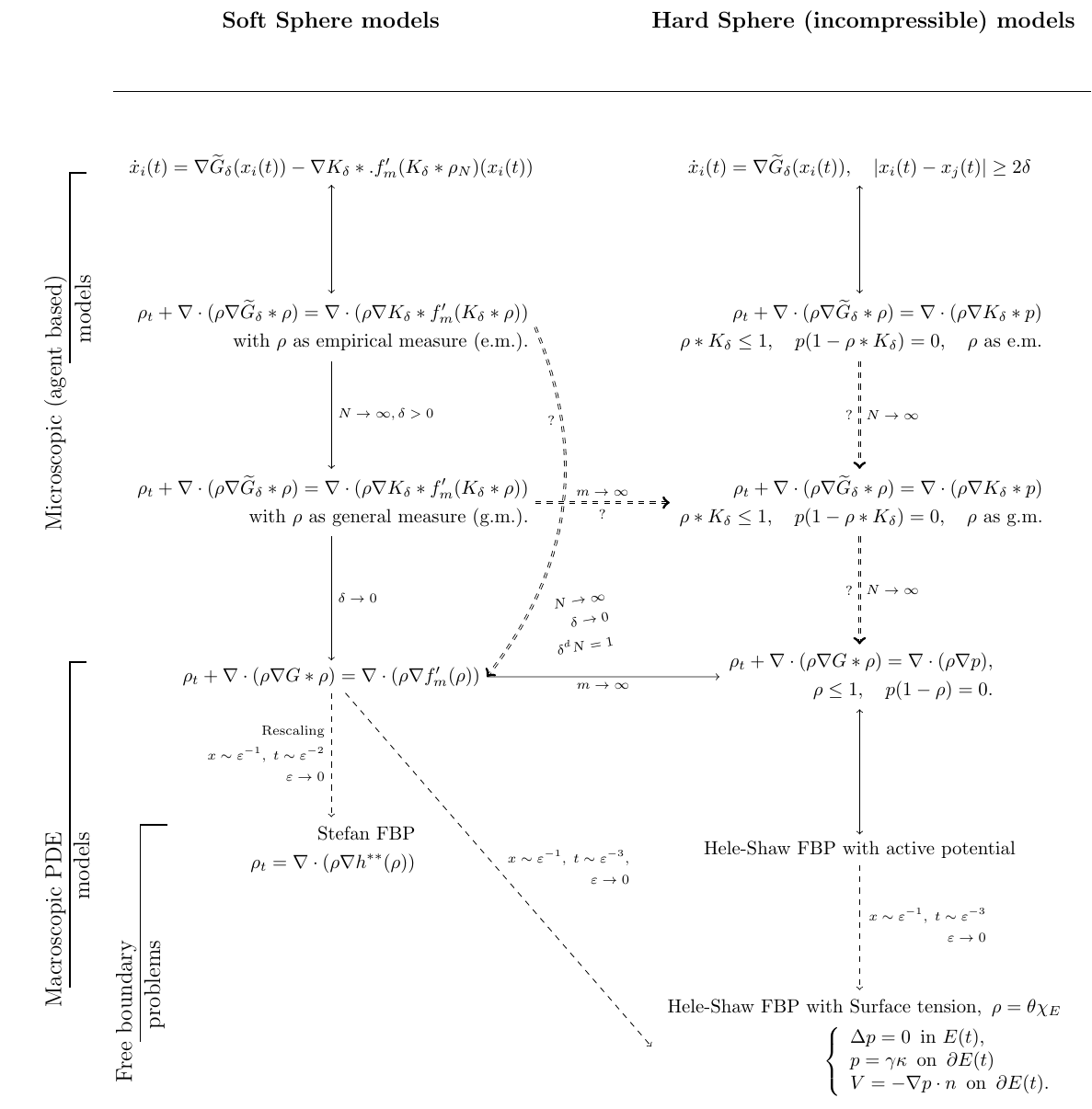}		
\caption{The various models discussed in this paper and their relationships. The dashed arrows indicate conditional convergence results. The double dashed arrows indicate open problems}
\end{figure}

\section{Energies and gradient flows.}
Formally,  all the models described in this paper have a gradient flow structure on the space of probability measures with finite second moment 
$$
\mathcal P_2(\R^d) := \left\{ \rho\in \mathcal P(\R^d)\, ;\, \int_{\R^d} |x|^2 d\rho(x) <\infty\right\}
$$
endowed with the $2$-Wasserstein metric which we denote by $d_W$.
This structure was first formalized for the Fokker-Planck equation in the seminal work of Jordan, Kinderlehrer and Otto \cite{JKO} and extended to other PDE in various work, see for instance \cite{CMV,AGS08,CDFLS11}.
Every equation presented in the introduction can  be written in the form
\begin{equation}\label{eq:cont}
\pa_t \rho + \div(\rho v ) =0,  \qquad v = -\na \frac{\delta \E}{\delta \rho} (\rho)\mbox{ in } (0,\infty)\times \R^d,
\end{equation}
for some energy functional $\E:\P_2 \to \R\cup\{+\infty\}$, which in turn provides a natural  Lyapunov functional for the evolution equation: Weak solutions  of \eqref{eq:cont}  satisfy the energy dissipation inequality
$$
\E[\rho(t)] + \int_0^t \int_{\R^d} |v|^2 d\rho \, ds \leq \E[\rho(0)].
$$
We briefly present below the energy functionals associated to our models. Throughout the paper
we will rely on these functionals and their asymptotic behavior to develop heuristic arguments and to justify singular limits of the associated PDEs. For most of this paper, we will work with measures $\rho$ which are absolutely continuous with respect to Lebesgue measure. This is denoted by $d\rho \ll dx$ and we will use $\rho$ to denote both the measure and its density with respect to Lebesgue measure.

\medskip

The (macroscopic soft-sphere) aggregation-diffusion equation \eqref{eq:LSS} is classically associated with the energy
\begin{equation}\label{eq:Ef}
\E_f [\rho] := \F_f [\rho] - \frac{1}{2}\iint_{\R^d\times\R^d} G(x-y)\, d\rho(x) \, d\rho(y),  \quad \F_f[\rho] := 
\begin{cases}
\displaystyle	\int_{\R^d} f(\rho(x))\, dx, 	&\text{if }\rho\in\P_2(\R^d), \; d\rho \ll \, dx, 	\\[4pt]
	+\infty, 	&\text{otherwise}.
\end{cases}
\end{equation}
We use the notation $\E_m$ (resp. $\E_\infty$) and 
$\F_m$ (resp. $\F_\infty$) when $f$ is given by \eqref{eq:fm} (resp. \eqref{eq:finfty}). In particular, for $\rho \in \P_2(\R^d)$ such that $d\rho \ll \, dx$, we have
$$
\F_m[\rho] = \int_{\R^d}  \frac{\rho(x)^m}{m-1} \, dx, \qquad \F_\infty[\rho]
= \left\{
\begin{array}{cl}
	0, 	&\mbox{ if }    \rho\le 1 \text{ a.e.} 	\\
	+\infty, 	&\text{ otherwise}
\end{array}
\right.
$$
and the energy $\E_\infty$, given by
\begin{equation}\label{eq:Einfty}
\E_\infty [\rho]  =\F_\infty[\rho]- \frac{1}{2}\iint G(x-y)\, d\rho(x) \, d\rho(y)  ,
\end{equation}
is associated with the hard-sphere model \eqref{eq:LHS}.

Similarly, the microscopic model \eqref{eq:NLSS} is a gradient flow for the regularized energy
\begin{equation}\label{eq:Edelta}
\E_{\delta}[\rho]:= \F_f [\rho*K_\delta]    - \frac{1}{2}\iint_{\R^d\times\R^d} \widetilde G_\delta   (x-y)\, d\rho(x) \, d\rho(y).
\end{equation}
Notice that $\E_\delta[\rho] = \E_f[K_\delta * \rho]$. In particular, thanks to the convolution with $K_\delta$, the measure $\rho$ does not have to be absolutely continuous with respect to the Lebesgue measure for $\E_\delta$ to be finite.
These energies play an important role in our proofs, and the $\Gamma$-convergence of $\E_{\delta}$ to $\E_f$ as $\delta\to 0$ and that of $\E_{m}$ to $\E_\infty$ as $m\to\infty$ are  good indicators of the relationship between the corresponding PDEs. 
\medskip

But such $\Gamma$-convergence results play an even more central role in the second part of the paper and the derivation of the free boundary problems \eqref{eq:stefan0} and \eqref{eq:HS}.
First, we note that the rescaled equation  \eqref{eq:LSSeps} is a gradient flow for the rescaled energy
\begin{equation}\label{eq:Jeps0}
 \F_f [\rho]  - \frac{1}{2}\iint_{\R^d\times\R^d} G_\eps(x-y)\, d\rho(x) \, d\rho(y), \qquad G_\eps(x)=\eps^{-d} G(\eps^{-1}x).
\end{equation}
Up to a constant, this energy can also be written as
\begin{equation}\label{eq:Jeps}
\J_\eps[\rho]  =  
\begin{cases}
\ds \int_{\R^d} h(\rho) \, dx +\frac{1}{4} \iint_{\R^d\times \R^d }  G_\eps(x-y)[\rho(x)-\rho(y)]^2\, dx\,dy 
& \ds \text{ if }  \rho\in\P_2(\R^d), \; d\rho \ll \, dx\\
\infty & \mbox{ otherwise.}
\end{cases}
\end{equation}
where $h$ is the double well-potential defined by \eqref{eq:hh} (with wells at $\rho=0$ and $\rho=\theta$).
This functional is a non local approximation of the Allen-Cahn energy 
\begin{equation} \label{eq:ACdf}
 \int_{\R^d}  h(\rho(x))\, dx + \eps^2 \beta \int_{\R^d} \frac 1 2 |\na \rho|^2\, dx,
\end{equation}
commonly used in problems that exhibit phase separation.
We will prove 
in particular that 
as long as $\int_{\R^d} (1+ |x|^2) G(x)\, dx <\infty$, 
$\J_\eps$ $\Gamma$-converges to (see Proposition \ref{prop:Gamma1}):
\begin{equation}\label{eq:J*}
\J^*[\rho ]=\int_{\R^d} h^{**}(\rho) \, dx,
\end{equation}
whose corresponding gradient flow is the generalized Stefan free boundary problem \eqref{eq:stefan0}.
Characteristic functions $ \theta\chi_E$ are stationary solutions of \eqref{eq:stefan0} and 
$\J_\eps(\theta\chi_E) \to \J^*(\theta\chi_E)=0$ as $\eps\to0$. But the rescaled functional $\G_\eps[\rho]:=\eps^{-1}\J_\eps[\rho]$ $\Gamma$-converges to the perimeter functional (see Theorem \ref{thm:Gamma})
$$
\G_0(\rho) = 
\begin{cases}
\gamma \theta  P(E) & \mbox{ if } \rho\in \P_2(\R^d) ,\quad  \rho = \theta \chi_E \in \BV(\R^d)\\
\infty & \mbox{otherwise}
\end{cases}
$$
for some constants $\gamma$ depending on $\sigma$ and $\theta$ where the perimeter $P(E)$ of the set $E\subset \R^d$ is defined by
\begin{equation}\label{eq:Per}
P(E) = \sup\left\{\int_E \div g\, dx\, ;\, g\in C^1_0( \R^d ; \, \R^d), \quad |g(x)|\leq 1, \; \forall x\in \R^d\right\}.
\end{equation}
Such results, for general kernels $G$ and smooth double-well potentials were first derived in~\cite{AB98}. 
This convergence is also proved by a different approach in \cite{M23} when $G$ is given by \eqref{eq:G} and $f$ is given by \eqref{eq:fm} with $m>2$ and in \cite{MW,KMW2} when $f=f_\infty$. 
In these papers, the role of boundary conditions (when the problem is set  in a bounded domain) plays a crucial role 
(see Theorem \ref{thm:4}) and led to the derivation of contact angle conditions.

The final observation is that the Wasserstein gradient flow for the perimeter functional is the Hele-Shaw free boundary problem with surface tension \eqref{eq:HS} (see \cite{OttoL}). The $\Gamma$-convergence of $\G_\eps$ to $\G_0$ is thus  consistent with the convergence of \eqref{eq:LSSeps} to \eqref{eq:HS}.
However, the $\Gamma$-convergence of the energy alone does not guarantee the convergence of the corresponding gradient flows. Additional sufficient conditions for such a convergence were established by Sandier-Serfaty \cite{serfaty} and used, for example, to justify the convergence of \eqref{eq:KellerSegel} to \eqref{eq:LHS} in the limit $m\to\infty$ in \cite{CT20}. 
These additional conditions can be difficult to establish in other frameworks and we will not review or extend such results in this paper.
For the derivation of  the free boundary problems \eqref{eq:stefan0} and \eqref{eq:HS} we will take a different approach and state conditional convergence results (see Theorems \ref{thm:stefan} and \ref{thm:HS}), which imply the convergence of the solutions under an assumption on the convergence of the energy (see \eqref{eq:EAS}). Such conditional convergence results have a long history with free boundary problems involving mean-curvature (see \cite{OttoL,LO16,EO15,JKM}), and proving rigorous convergence results without this assumption remains an important and challenging problem.

 \medskip

\medskip

\part{\Large Microscopic/macroscopic soft-sphere and hard-sphere models}

\section{Soft-sphere models: The blob method}\label{sec:blob}
The main goal of this section is to derive the microscopic model \eqref{eq:NLSS}  from particle dynamics and review some of the recent literature on this and related models.
The convolution with the kernel $K_\delta$ is the key feature of this model. It accounts for the finite size ($\delta>0$) of the particles, as opposed to treating them as points. 
We will see that  \eqref{eq:NLSS}  has solutions in the form of empirical distributions (sums of Dirac masses centered at points $x_i(t)$), a fact that allows us to connect \eqref{eq:NLSS}  to the corresponding microscopic model describing the motion of individual particles via a system of coupled ODEs.

From a mathematical perspective,  \eqref{eq:NLSS} is the gradient flow of a $\lambda$-convex energy functional (under appropriate assumptions on $K_\delta$ and $f$). This classical framework yields the well-posedness of \eqref{eq:NLSS} in the set of probability measures.
Finally, we will discuss the connection between \eqref{eq:NLSS} and the classical aggregation-diffusion equation \eqref{eq:LSS} in the limit $\delta\to0$.

\subsection{Transport equation with interaction potential}
Transport equations without any diffusion and with sufficiently smooth interaction potential 
$\widetilde G:\R^d\to \R$ (satisfying $\na \widetilde G(0)=0$) are naturally associated to deterministic particle systems via the empirical distribution:
Denoting by $\{x_i(t) \}_{i=1,\dots, N}$ the positions of $N$ particles, solutions to the following system of ODEs
\begin{equation}\label{eq:ODE1}
\dot{x}_i (t) 
 =  \frac 1 N \sum_{j=1}^{N} \na \widetilde G (x_j(t)-x_i(t)) \qquad \forall i=1,\dots, N,
\end{equation}
the corresponding empirical distribution 
$
\rho_N (t,x)  = \frac 1 N \sum_{j=1}^{N} \delta(x-x_j(t))
$ 
is the solution (in the sense of distribution) of the continuity equation
\begin{equation}\label{eq:continuity}
\pa_t \rho + \div(\rho \na \widetilde G*\rho)=0
\end{equation}
with initial data $\rho_N(0,x) = \frac 1 N \sum_{j=1}^{N} \delta(x-x_j(0))$ (see for instance \cite{CDFLS11,CCH14}).

Equation \eqref{eq:continuity} appears in many settings, particularly in the mathematical modeling of collective behaviors.
We refer to  \cite{CDFLS11} and the many references therein for a detailed discussion of the associated well-posedness theory. We recall that \eqref{eq:continuity} is a gradient flow for the energy
$$
-\frac{1}{2}\int_{\R^d} \!\!\int_{\R^d} \widetilde G(x-y) \rho(x)\rho(y)\, dx\, dy
$$
defined on $\mathcal P_2(\R^d)$ (this was formalized in \cite{CMV} following ideas that were first introduced in \cite{Otto01}).
In that context, a key assumption for the well-posedness of \eqref{eq:continuity} is the $\lambda$-convexity of the kernel $-\widetilde G$ for some $\lambda \leq 0$ (see \cite{CDFLS11}).
Under this assumption, we also get the following stability estimate 
\begin{equation}\label{eq:Dob}
d_W(\rho^1(t),\rho^2(t))\le d_W(\rho^1(0),\rho^2(0)) e^{-\lambda t}
\end{equation}
for any two weak measure solutions $\rho^1(t)$ and $\rho^2(t)$ of \eqref{eq:continuity} in $ C([0,\infty);\mathcal P(\R^d))$.
Stability estimates have a long history and a similar inequality was first proved by 
Dobrushin \cite{D79} (see also Golse \cite{G16}) with the $1$-Wasserstein distance. 
Deriving such estimates (which can lead to useful numerical approximations) when the interaction potential $\widetilde G$ is less regular and possibly in different topologies
is still an active area of research. 
Inequality \eqref{eq:Dob} implies in particular that if $\rho_N(0)$ converges to $\rho_{in}$ (with respect to $d_W$), the  measure $\rho_N(t,x)$ converges to the corresponding solution of \eqref{eq:continuity}.
We will   say that \eqref{eq:continuity} is the mean-field limit of the ODE system \eqref{eq:ODE1}.

The mean field limit (but not \eqref{eq:Dob}) can also be justified  for some more singular potentials (see for instance  \cite{CCH14} when $\widetilde G(x)\sim |x|^b$ with $b>2-d$), but not, to our knowledge, with the Newtonian kernel as in the PKS model.
Other classical references for  mean-field limit of related model include Sznitman~\cite{S91} and Golse~\cite{G16} as well as the recent reviews with both theoretical and numerical perspectives~\cite{CD22a,CD22b}. 
\medskip
  
\subsection{Linear/Nonlinear repulsion}
When diffusion (linear or nonlinear) is added to \eqref{eq:continuity}, solutions starting with empirical measures are instantly regularized by the diffusion.
This makes the approximation of the corresponding PDE by a deterministic ODE system challenging (see \cite{DM90,Russo}).
A classical way to deal with this issue is to consider instead a system of stochastic ODEs (adding a Brownian motion to \eqref{eq:ODE1}) as is done with the discretization of the linear Fokker-Planck equation by Langevin dynamics.
Some nonlinear diffusion equations  have also been handled via such methods \cite{O01,MCO05,P07,FP08}. We also refer to some recent work treating singular potentials via the so-called \textit{relative-entropy method}~\cite{Jabin2017,JW18,BJW19} and \textit{modulated energy method}~\cite{RS23} (the result of \cite{BJW19} applies to the Patlak-Keller-Segel model in dimension~2).

We will now describe a different approach, based on the so-called {\it blob method}, which is 
a deterministic approach in which the (nonlinear) diffusion is approximated by a nonlocal interaction term which preserves the empirical distribution (see \cite{CB16,CCP19,LM01,CEHT23,BE23,carrillo2023nonlocal}). 
This approach was also used  by  Motsch and Peurichard in \cite{MP18} to derive a model for tumor growth.
 
\subsection{Nonlocal nonlinear repulsion}
The {\it blob method} amounts to taking into 
 account the finite size $\delta>0$ of the particles (at the microscopic level) and some volume exclusion principle.
 A key features of the microscopic models we have in mind is a repulsive effect with finite radius: Particles that are within a certain distance of each others feel the effect of a strong repulsion, which prevents the density of particles from becoming too large.

Given a compactly supported function $\bar K_\delta:\R^d \to [0,\infty) $, we introduce the  counting function 
$$\mu_N(x) =\sum_{j=1} ^N \bar K_\delta(x-x_j(t)).$$
When $\bar K_\delta=\chi_{B_\delta(0)}$ (an important example we will come back to regularly), $\mu_N(x)$ counts how many particles are located in $B_\delta(x)$. 
If we think of the particles as {\it hard-spheres} of radius $\delta$, it is natural to impose
 the constraint $\mu_N\leq 1$. We will discuss this setting, and its many challenges, in the next section,  but for now we will consider  a {\it soft-sphere} model which instead of this hard constraint includes a repulsion force  that drives the particles away from crowded regions with a strength that depends on~$\mu_N$.

Before introducing the model, we note that we can interpret  the quantity $\int \bar K_\delta(x)\, dx=|B_\delta| = \omega_d \delta^d$ as the volume occupied by one particle, so the total volume occupied by all particles is equal to $m_0:=N \int \bar K_\delta(x)\, dx$.
Introducing the normalized kernel $K_\delta = \frac{1}{\int \bar K_\delta(x)\, dx} \bar K_\delta$, we can write 
$$ \mu_N (x) = N \bar  K_\delta *\rho_N(x) = m_0  K_\delta *\rho_N(x) .$$
In order to simplify the notations, we will take $m_0=1$ below so that the counting function $\mu_N$ can be written as 
$$ \mu_N (x) =  K_\delta *\rho_N(x) , \qquad \int K_\delta(x)\, dx =1, \quad \supp K_\delta = B_\delta.$$
Note that we  can choose $K_\delta(x)=\frac{1}{\omega_d \delta^d}K(\frac x \delta)$ with $K(x)=\chi_{B_1}(x)$, but when convenient we will replace this kernel with a smooth radially symmetric decreasing function supported in $B_1$. 
In what follows $N$ and $\delta$ will be treated as independent parameters, but we should remember that this interpretation of the microscopic model imposes the constraint $N\sim \delta^{-d}$.

\medskip

The soft-sphere model is obtained by adding a repulsive term to    \eqref{eq:ODE1} as follows:
\begin{equation}
	\label{eq:char_particle}
\dot{x}_i (t) 
 =  \frac 1 N \sum_{j=1}^{N} \na \widetilde G (x_j(t)-x_i(t)) 
 - \int  K_\delta (x_i(t)-y) \na f'(\mu_N) (y)\, dy, \qquad i=1,\dots N,
\end{equation} 
for a convex function $f$ (for instance given by \eqref{eq:fm}). Equation \eqref{eq:char_particle}
 can also be written as 
\begin{equation}\label{eq:blob}
\dot{x}_i (t)  =  \na \widetilde G *\rho_N(x_i(t)) - \na K_\delta * f'(  K_\delta *\rho_N(x)) (x_i(t)),\qquad i=1,\dots N.
\end{equation}
The particular structure of this repulsive term, with the double convolution, is convenient from a mathematical view point (it leads to an interpretation of the corresponding  PDE as a gradient flow).
From a modeling view point, it can be interpreted as follows (when $K(x)=\chi_{B_1}(x)$): 
Since a particle is represented by the ball $B_\delta(x_i(t))$, the repulsion force exerted on its center $x_i(t)$ is the  average  over $B_\delta(x_i(t))$  of the gradient of the pressure
$f'(\mu_N(x))$  - which depends on the counting function $\mu_N$. 
\medskip

System \eqref{eq:blob} is similar to the deterministic system used for example in 
\cite{CB16,CCP19,CEHT23,carrillo2023nonlocal} to approximate non-linear diffusion equation via what has been called  the blob method.
One advantage of this approach is the fact that 
given a solution $\{ x_i(t)\}_{i=1,\dots, N}$ of \eqref{eq:char_particle}, the empirical distribution  $\rho_N(t,x)$ is a solution (in the sense of distribution) of the nonlinear transport equation
\begin{equation}
	\label{eq:double_conv}
\pa_t \rho + \div(\rho \na \widetilde G*\rho ) =\div( \rho  \na K_\delta*f'( K_\delta*\rho)).
\end{equation}

However, it has been observed (see for example \cite{MP18}) that the long-time behavior of~\eqref{eq:double_conv} exhibits discrepancies with that of the microscopic dynamics~\eqref{eq:blob}:
When $\widetilde G$ is zero and $K_\delta$ is compactly supported, numerical simulations show that solutions of \eqref{eq:double_conv} with absolutely continuous densities spread and converge (locally) to zero (while preserving the mass) as $t\to \infty$ 
(which is consistent with the behavior of solutions of the porous media equation)
whereas, for a finite number of particles $N$, the solutions to~\eqref{eq:blob} converge to a finite non-zero value (in particular, the dynamics stops when all the points $x_i(t)$ are at distance greater than $2\delta$ from one another).  
So, while we have a one-to-one correspondence between the solutions of \eqref{eq:blob} and empirical measures 
solutions of \eqref{eq:double_conv}, these empirical measures are \textit{unstable} solutions of \eqref{eq:double_conv} (in the sense that perturbations of the Dirac masses diffuse in space under the dynamics of \eqref{eq:double_conv}, hence departing further from their original Dirac structure). 
This discrepancy between the microscopic (ODE) and macroscopic (PDE) dynamics is not unique to this model and we shall see this again for  the hard-sphere (or hard congestion) models in the next section
\footnote{Motsch and Peurichard \cite{MP18} proposed a stabilizing method which consists in
assuming that $f(\rho)=0$ for $\rho\leq \rho^*$. This is to ensure that no repulsion occurs when the particle density $\rho$ falls below some threshold value $\rho_*$.}.

\medskip
Note that \eqref{eq:double_conv} is reminiscent of the transport equation with Brinkman's law, which can be written as 
\begin{equation}\label{eq:brink}
\pa_t \rho = \div(\rho \na \widetilde G*\rho ) + \div( \rho  \na K_\delta*f'(\rho)), \qquad K_\delta - \delta \Delta K_\delta = \delta_{x=0}.
\end{equation}
This model has a similar behavior as  \eqref{eq:double_conv} when $f'(\rho)=\rho$, but when $f(\rho)$ is given by \eqref{eq:fm} with $m\gg1 $, \eqref{eq:brink} imposes  a height constraint on $\rho$ while \eqref{eq:double_conv} imposes a constraint on $K_\delta*\rho$. Both the dynamics and the mathematical theory are quite different in that case.

\medskip

\subsection{A nonlocal, nonlinear Keller-Segel model}
In the parabolic-elliptic approximation of the Patlak-Keller-Segel model, the attractive potential $\phi = G*\rho$ solves $\sigma \phi-\eta \Delta \phi =  \rho$ and represents the concentration of a self-emitted chemo-attractant substance.
At the microscopic scale, we should also take into account the finite size of the particles $B_\delta$ (here, we take $K=\chi_{B_1}$ again):
First, the production of the chemo-attractant occurs over the whole ball $B_\delta(x_i)$ instead of being concentrated at the center $x_i$, leading to 
\begin{equation}\label{eq:cdelta}
\sigma \phi_\delta-\eta \Delta \phi _\delta = K_\delta * \rho_N
\end{equation}
Next, the particles sensors are also located over the whole ball $B_\delta(x_i)$, rather than only in the center $x_i$, so that 
the  velocity is given by  $ \na K_\delta *\phi_\delta (x_i(t))$:
\begin{equation}\label{eq:PKSmicro}
\dot{x}_i (t) 
 =   \na K_\delta *\phi_\delta(x_i(t)) - \na K_\delta * f'(  K_\delta *\rho_N(x)) (x_i(t)) 
\end{equation}
This system is exactly  \eqref{eq:char_particle} if we replace $\widetilde G$ with $\widetilde G_\delta   :=K_\delta * \widetilde G* K_\delta$.
The good news is that while $ G$ is too singular to apply the existing well-posedness theory, the regularized kernel $\widetilde G_\delta   $ is much nicer. In particular\footnote{Indeed, the function $u=  G* K_\delta$ solves $\sigma u -\eta\Delta u = K_\delta\in L^2$ and is thus in   $H^2$. We deduce $ \|D^2 \widetilde G_\delta    \|_{L^\infty} = \| K_\delta*D^2u \|_{L^\infty} \leq \|K_\delta\|_{L^2} \|D^2u\|_{L^2} \leq C  \|K_\delta\|_{L^2}^2$.},  we have $\widetilde G_\delta     \in W^{2,\infty}$ as soon as $K_\delta \in L^2$.
\medskip

Combining the results mentioned above (for instance \cite{CDFLS11} and  \cite{carrillo2023nonlocal}),
we deduce that when $\{x_i(t)\}_{i=1,\dots, N}$ solves
\eqref{eq:cdelta}-\eqref{eq:PKSmicro}, the empirical distribution $\rho_N(t,x)$ solves \eqref{eq:NLSS} which is the gradient flow in $\mathcal P_2(\R^d)$ for the energy functional \eqref{eq:Edelta}.
This gradient flow is studied in  \cite{CDFLS11} when $f=0$ (in particular $\widetilde G_\delta   $ satisfies the assumptions of  \cite{CDFLS11}  when $K_\delta\in L^2$) and in  \cite{carrillo2023nonlocal} when $\widetilde G_\delta   =0$ and $f$ satisfies some assumptions satisfied in particular when $f(\rho) = \frac{1}{m-1}\rho^m$ (but with additional assumption on $K_\delta$).
We also refer to earlier work in a similar spirit by Carrillo, Craig, and Patacchini~\cite{CCP19}. The special case $m=2$ has been given detailed attention by Burger and Esposito~\cite{BE23} and Craig,  Elamvazhuthi,  Haberland and  Turanova~\cite{CEHT23} (with applications in cross-diffusion and sampling, respectively). Adapting the method developed in these papers (in particular \cite{CDFLS11,CCP19,CEHT23}), we get:
\begin{theorem}\label{thm:stab}
Fix $\delta>0$, assume that $K\in C^2(\R^d)$, $K\in W^{2,\infty}(\R^d)$, and take $f=f_m $ for some $m>1$.
For all $\rho_{in} \in \mathcal P_2(\R^d)$ with $\E_\delta[\rho_{in}]<\infty$, there exists a unique gradient flow solution $\rho^\delta(t)$  in $AC_{loc}([0,\infty);\mathcal P_2(\R^d))$ of \eqref{eq:NLSS}.
This solution is characterized by:
\begin{equation}\label{eq:PDE}
\begin{cases}
\pa_t \rho^\delta + \div(\rho^\delta \na \widetilde G_\delta   *\rho^\delta ) = \div( \rho^\delta  \na K_\delta*f_m'( K_\delta*\rho^\delta)) \quad \mbox{ in } \mathcal D'( (0,\infty)\times\R^d)\\
\lim_{t\to 0^+} d_W(\rho^\delta(t), \rho_{in}) = 0.
\end{cases}
\end{equation}
Furthermore, given 
 two such solutions $\rho^1(t)$ and $\rho^2(t)$, the following stability estimate holds
\begin{equation}\label{eq:Stability2}
d_W(\rho^1(t),\rho^2(t))\le d_W(\rho^1(0),\rho^2(0)) e^{-\lambda_0 t}
\end{equation}
with  
\begin{equation}\label{eq:lambda0}
 \lambda_0:= - C(\| K\|_{W^{2,\infty}(\R^d)}) \delta^{-d(m-1)-2} + C(\|K\|_{L^2}) \delta^{-d}.
 \end{equation}
\end{theorem}

Inequality \eqref{eq:Stability2} implies the mean-field limit for fixed $\delta$: if the empirical measure converges to a density $\rho_{in}$, in  the sense that
$d_W(\rho_{N} (t=0), \rho_{in} )\to0$, then the empirical distribution $\rho_N(t,x)$ converges weakly to the solution $\rho$ of~\eqref{eq:PDE}. The constant $\lambda_0$ given by \eqref{eq:lambda0} is a (lower bound for the) modulus of convexity for the energy functional $\mathscr{E}_\delta$ (defined by \eqref{eq:Edelta}) with respect to the underlying optimal transport geodesic structure on $\mathcal{P}_2$.
Importantly, we note that when $\delta\to0$,  we have $\lambda_0 \to -\infty$ as $\delta\to 0$. On the other hand, it is known that the limiting energy $\mathscr{E}$ is convex. This discrepancy suggests that this \eqref{eq:lambda0} is far from optimal.

\begin{proof}[Sketch of the proof of Theorem \ref{thm:stab}]
Our assumptions on $K_\delta$ and the properties of $\widetilde G_\delta   $ ensures that $\E_{\delta}:\mathcal P_2(\R^d)\to (-\infty,+\infty]$ is 
proper, lower-semicontinuous and
$\lambda$-convex with modulus of convexity given by \eqref{eq:lambda0} (see  \cite{CDFLS11,CCP19}).
The differentiability of $\E_{\delta}$ can be proved using the same arguments as in \cite{CDFLS11} (for the $\widetilde G_\delta   $ term) and \cite{CCP19} (for the $f_m$-term) and leads to
$$
\frac{\delta \E_{\delta}}{\delta\rho}[\rho] = - \widetilde G_\delta   *\rho + f_m'(K_\delta * \rho)*K_\delta.
$$
Existence and uniqueness of the gradient flow $\rho^\delta$, as well as the fact that the gradient flow is a curve of maximal slope, follows from \cite[Theorem 11.2.1]{AGS08}.
Furthermore, $\rho^\delta$ solves the continuity equation
$$\pa_t \rho^\delta + \div (\rho^\delta v^\delta) = 0 , \quad \mbox{ in  } \mathcal D'(\R^d\times(0,\infty))
$$
with velocity
$$ v^\delta(t) = \na \widetilde G_\delta   *\rho^\delta(t)- \na K_\delta * f_m'(K_\delta*\rho^\delta(t)) \qquad \mbox{a.e. } t>0$$
and satisfies the energy inequality
\begin{equation}\label{eq:energydeltaine}
 {{\E_\delta}}[\rho^\delta(t)] + \int_0^t\int_{\R^d}  |v^\delta|^2 d\rho^\delta(s)\, ds \leq  {\E_\delta}[\rho_{\text{in}}].
\end{equation}
We note that the presence of the convolution  and the assumption on $K_\delta$  implies that $v^\delta$  is uniformly bounded (and $C^1$), so $\rho^\delta$ solves \eqref{eq:PDE}. Conversely, this also implies  (see  \cite[Proposition 3.12]{CEHT23}) that any solution of \eqref{eq:PDE} is in fact in $AC_{loc}([0,\infty);\mathcal P_2(\R^d))$ and is a gradient flow of $\E_\delta$.

Finally we get   \eqref{eq:Stability2} as a consequence of the usual stability estimate for $\lambda$-gradient flows  \cite[Theorem 11.2.1]{AGS08} (see also \cite[Theorem 1.4]{CEHT23}).
\end{proof}
For future references, we also recall that $\rho^\delta$ satisfy the propagation of the second moment:
\begin{equation}\label{eq:second}
\int_{\R^d} |x|^2 \rho^\delta(t,x)\,dx \leq \left(\int_{\R^d}|x|^2 \rho_{\text{in}}(x)\,dx + \mathcal{E_\delta}[\rho_{\text{in}}]\right) e^t.
\end{equation}
which follows from   \eqref{eq:energydeltaine} since we have
$$
\frac{d}{dt} \int_\Omega |x|^2 \rho^\delta(t,x)\, dx = 2 \int_\Omega x \cdot v^\delta \rho^\delta\, dx 
\leq 2\left(\int _\Omega |x|^2 \rho^\delta\, dx\right)^{1/2} \left( \int _\Omega |v^\delta|^2 \rho^\delta\, dx\right)^{1/2} .
$$

\medskip

\subsection{Convergence to the macroscopic model: The limit $\delta\to 0$}
We now turn to  the question of the convergence of the nonlocal model \eqref{eq:NLSS} to the classical aggregation-diffusion equation \eqref{eq:LSS} when $\delta\to0$.
 We note that such a convergence cannot hold unless \eqref{eq:LSS} has a solution, and it is well-known that aggregation diffusion equations can experience finite time blow-up. 
When $f$ is given by \eqref{eq:fm} and $G$ is the chemotaxis potential \eqref{eq:G}, then global in time (bounded) solutions exist in the subcritical regime $m>m_c:=2-\frac 2 d$ (more generally, if $G(x)\sim \frac{1}{|x|^k}$, then the critical power is given by $m_c:=1+\frac kd$, see \cite{CCY19}).
In fact, we have:
\begin{theorem}\label{thm:existenceLSS}
Let $f=\frac{1}{m-1}\rho^m$ with $m>2-\frac 2 d$ and $G$ be given by \eqref{eq:G}. Given $\rho_{in}\in \mathcal P_2(\R^d)$ such that $\| \rho_{in}\|_{L^\infty(\R^d)}<\infty$ and $\supp \rho_{in} \in B_R$,
equation \eqref{eq:LSS} has a unique global solution $\rho(t,x)$ with initial data $\rho_{in}(x)$.

Furthermore,  for all $T>0$ there exist $C(T)$ and $R(T)$ such that
$$ \| \rho\|_{L^\infty((0,T)\times\R^d)}\leq C(T), \qquad \supp \rho(t) \subset B_{R(T)} \quad \forall t\in[0,T].
$$ 
\end{theorem}
We refer for example to \cite{Sugiyama,K05,CC06,BRB} for the existence of a global bounded solution. The propagation of the $L^\infty$ norm and compact support can also be proved as in \cite[Appendix A and B]{HLP} and uniqueness is proved in \cite{BRB}.
\medskip

While this result only requires $m>2-\frac 2 d$, we will now make the  stronger assumption  $m>2$ to simplify the analysis (the assumption $m>2$ will be crucial in the second part of this paper but the result below should  hold for all $m>2-\frac 2 d$). 
Indeed, since $\int_{\R^d} G(x)\, dx = \frac 1 \sigma$, we can  rewrite the energy \eqref{eq:Ef} as
\begin{align*}
\E_m[\rho] 
& =  \int_{\R^d} f_m(\rho) \, dx - \frac{1}{2}\iint_{\R^d\times\R^d} G(x-y)\, \rho(x) \, \rho(y) \, dx\, dy\\
& =  \int_{\R^d} f_m(\rho)  -\frac 1 {2\sigma} \rho^2  \, dx + \frac{1}{2}\iint_{\R^d\times\R^d} G(x-y)\, \left[\rho (x)^2-  \rho(x) \, \rho(y) \right] dx\, dy\\
& =  \int_{\R^d} f_m(\rho) -\frac 1 {2\sigma} \rho^2 \, dx + \frac{1}{4}\iint G(x-y)[\rho(x)-\rho(y)]^2 \, dx\,  dy.
\end{align*}
When $m>2$, the function $\rho\mapsto f_m(\rho) -\frac 1 {2\sigma} \rho^2  $ is bounded below. Since the equation preserves $\int \rho(t, x)\, dx$ over time, we can add a term $a\rho$ to make it non-negative  (see \eqref{eq:hh}).
This leads to the function $\rho\mapsto h_m(\rho)$ defined by \eqref{eq:hm},
which satisfies $h_m(\rho)\geq 0$ for all $\rho\geq 0$.
We can then work with the energy functional
$$
\overline{\E_m }[\rho] 
:=
\begin{cases}
\displaystyle \int  h_m(\rho(x))\, dx + \frac{1}{4}\iint G(x-y)[\rho(x)-\rho(y)]^2 \, dx\,  dy & \text{ if }\rho\in\P_2(\R^d), \;d\rho(x) \ll \, dx  \\
\infty & \mbox{ otherwise }
\end{cases}
$$
which differs from $\E_m $ by a constant. 
Similarly, we can rewrite the energy \eqref{eq:Edelta}, up to a constant, as
$$
\overline{\E_{\delta}}[\rho] 
:= 
\int   h_m(K_\delta*\rho(x))\, dx + \frac{1}{4}\iint G(x-y)[K_\delta*\rho(x)-K_\delta*\rho(y)]^2 \, dx\,  dy , \qquad \rho \in \mathcal P_2(\R^d).
$$
In particular, the energy inequality \eqref{eq:energydeltaine} holds with $\E_\delta$ replaced by $\overline{\E_{\delta}}$.

We can now prove:
\begin{theorem}\label{thm:convdelta}
Assume that $G$ is given by \eqref{eq:G} and that $f(\rho) = \frac{1}{m-1}\rho^m$   with $m>2$ 
and let $\rho_{in}$ be as in Theorem \ref{thm:existenceLSS}.
Let $\rho^\delta(t,x)$  be  the solution of \eqref{eq:NLSS}  given by Theorem \ref{thm:stab}
and  $\rho(t,x) $ be the solution of \eqref{eq:LSS} given by Theorem \ref{thm:existenceLSS}.
Then, the entire sequence $\rho^\delta(t)$ converges narrowly to $\rho(t)$ (uniformly in $[0,T]$).
\end{theorem}
We provide a  proof of this result in Appendix \ref{app:conv} by adapting 
 the arguments developed in \cite{carrillo2023nonlocal}  to prove a similar results when $G=0$.

\medskip

Theorem \ref{thm:convdelta} does not address the convergence of the microscopic model \eqref{eq:blob} when $\delta\to 0$ since the condition $\overline {\E_m} [\rho_{in}]<\infty$ excludes empirical measures.
As explained in \cite{carrillo2023nonlocal}, this result can however be combined with \eqref{eq:Stability2}:
If we approximate the measure $\rho_{in}$ by a sequence of empirical measures such that $ d_W(\rho_{N,\delta} (t=0), \rho_{in} ) \leq \frac C N$ and $N = o\left( e^{-\frac{1}{\delta^{2+d(m-1)}}}\right)$, then using \eqref{eq:Stability2} and the convergence of $\rho^\delta$ to $\rho$, we can show that 
the empirical distribution $\rho_{N,\delta}(t,x)$ converges to $\rho(t,x)$ solution of \eqref{eq:LSS} with initial condition $\rho_{in}$. 
This restriction on $N$ is very far from the scaling $N\sim \delta^{-d}$ necessary to preserve the total mass in our microscopic interpretation of the  blob  model. Numerical evidence (see \cite{CEHT23,CB16}) suggests that a good level of approximation  holds for much smaller number of particles but justifying this limit is a challenging  open problem.

To summarize, 
equation \eqref{eq:double_conv} preserve the empirical measure and thus describe precisely the evolution of a system of particles evolving according to the microscopic model \eqref{eq:blob}. 
When $\delta\to0$, this equation \eqref{eq:NLSS} is an approximation of~\eqref{eq:LSS} (which does not preserve the Dirac structure of $\rho_N$ in its evolution). 

To the best of our knowledge, this  approach to approximating nonlocal diffusion equations by particle methods originates from the works of Mas-Gallic~\cite{MG87} for kinetic equations and Oelschl\"ager~\cite{O90} for the quadratic porous medium equation. We also mention~\cite{LM01, DM90, FP08, CDW23} in this direction. The convolution by ${K}_\delta$ alters the diffusion in the sense that particles remain particles: $\rho_N(t) = \frac{1}{N}\sum_{i=1}^N\delta_{x_i(t)}$ is a solution to the regularised equation where each particle $x_i(t)$ evolves according to the system of characteristics.

\medskip

\subsection{Other microscopic models with repulsion}  
A similar regularization approach  can be used, at least formally, to justify other macroscopic models. For example, Karper, Lindgren and Tadmor \cite{Tadmor} followed a similar approach and  proposed the following microscopic model for chemotaxis:
\begin{align*}
	\begin{split}
\dot{x}_i (t)
& =   \na \phi (x_i(t)) -  \frac{ \beta(\mu_N(x_i))}{\mu_N(x_i(t))} \na  \mu_N(x_i(t) ).
\end{split}
\end{align*}
with  $\beta(\mu) := \frac{\mu}{1-\mu}$.
The corresponding empirical distribution then solves
\begin{equation}
	\label{eq:tadmor_conv}
\pa_t \rho_N + \div(\rho_N \na  \phi  ) = \div\left( \frac{\rho_N}{ K_\delta*\rho_N} \beta( K_\delta*\rho_N) \na (K_\delta*\rho_N)\right)
\end{equation}
and in the limit $N\to\infty$ and $\delta\to0$, we formally obtain 
\begin{equation}
	\label{eq:tadmor}
\pa_t \rho + \div(\rho \na \phi) = \div \left( \frac{\rho}{1-\rho}\na \rho\right)
\end{equation}
which corresponds to \eqref{eq:LSS} with $f(\rho)  = (1-\rho)\ln(1-\rho)$.
The use of such singular pressure is a classical way to enforce a congestion constraint $\rho\leq 1$ (see for instance \cite{HP02}) and a model similar to \eqref{eq:tadmor} was derived in \cite{Lush} from a stochastic particle system with volume effect, by enforcing the fact that two particles (which are assumed to be  rigid squares in \cite{Lush}) cannot overlap.
Justifying the limit from \eqref{eq:tadmor_conv} to \eqref{eq:tadmor}
rigorously, in the spirit of Theorem \ref{thm:convdelta}, is, to the best of our knowledge an interesting open problem.

\medskip

Other microscopic models have been proposed to describe short range repulsion. 
For example, Fischer, Kanzler and  Schmeiser \cite{FKS}  consider a one-dimensional model of particles
 $x_0(t)<x_1(t)<\cdots <x_N(t)$ in which each agent $x_i$ only interacts with its two immediate neighbors $x_{i-1}$ and $x_{i+1}$ via a repulsive force which depends only on their distance.
They derive a nonlinear diffusion model (depending on the repulsive force).
The extension of this model to higher dimensions is highly non-trivial and is very much related to similar difficulties faced in the so-called `upwind' scheme for hyperbolic equations~\cite{DFR15,FT22} for continuity equations with non-linear mobilities.
Other topological models, in which interactions depend on the ``rank" of a neighbor rather than its distance, have been studied for example by Blanchet and Degond \cite{BDegond16,BDegond17}. 
Finally, let us point out that there are many models of pedestrian flows that also include the congestion constraint (see for instance \cite{H02,LW11,CGL12}).

\section{Hard-sphere (incompressible) models}\label{sec:HS}
We now turn our attention to  the hard-sphere (or incompressible) model.
At the microscopic scale, the particles, still represented by balls  of radius $\delta>0$, are  prevented from overlapping (congestion constraint). At the macroscopic scale, this will lead to the assumption that the population density cannot increase above a prescribed value. 

On the one hand, we will see that the relationship between microscopic and macroscopic models is a lot more complex in this case than in the soft sphere-case  of the previous section.
On the other hand, we will show that the macroscopic hard-sphere model can easily be derived as the singular limit $m\to\infty$ of the (macroscopic) soft sphere models.
In the first part of the discussion, we assume that the ``desired" velocity field is given  a priori, but we also discuss the well-posedness properties of this macroscopic model when this velocity field is the gradient of the attractive chemotaxis potential.

\subsection{The microscopic hard-sphere model}\label{sec:micro}
The microscopic model we now introduce has been extensively studied in the context of congested crowd motion, in particular by Maury, Santambrogio et al. (see \cite{MRS,MRSV,Maury11,FM}). In this context the domain $\Omega$ and the boundary condition (existence of ``exits") play a very important role; this aspect will not be discussed in this paper.

Unlike the soft sphere model, in which the velocity is modified by a repulsive force, we now   impose the hard-sphere constraint 
\begin{equation}\label{eq:constraint}
D_{ij}({\bf x} (t)):=|x_i(t)-x_j(t)|\geq 2\delta \qquad \mbox{ for all $i\neq j$, for all $t\geq0$}.
\end{equation}
In terms of the empirical measure $\rho_N (t,x)  = \frac 1 N \sum_{j=1}^{N} \delta(x-x_j(t))$, and using the notation of the previous section with $K=\chi_{B_1}$, we can also write this constraint as  
\begin{equation}\label{eq:density_nonlocal}
\mu_N = K_\delta * \rho_N \leq 1.
\end{equation}
As a consequence, the actual velocity of a particle will be the projection of the desired velocity onto the set of admissible velocity, namely velocities that preserve the constraint \eqref{eq:constraint}.

\medskip

Following \cite{MRSV}, we assume that the particle $i$ has a desired velocity $v_i $ which might depend on the positions of all the particles.
Below, we use the notation ${\bf x}= (x_1,\dots , x_N)\in (\R^{d})^N$ to denote the position of the $N$ particles, and ${ \bf v}({\bf x}) = (v_1({\bf x}), \dots , v_N({\bf x}))$ for their desired velocity.
The constrained microscopic model can then be written (see \cite{MRSV}):
\begin{equation}\label{eq:ODEu}
\dot{{\bf x}}(t) = {\bf u}({\bf x}),   \quad {\bf u}= P_{C({\bf x})} {\bf v}, 
\end{equation}
where $P_{C({\bf x})} {\bf v}$ denotes the projection of ${\bf v}$ onto the set  $C({\bf x})$ of admissible velocities.
If we define $D_{ij}({\bf x}) := |x_i-x_j|$, then admissible velocities must be such that $D_{ij}({\bf x}) $ cannot decrease if 
 the $i$ and $j$ particles are already touching. This leads to 
$$
C({\bf x}):= \{{\bf u}\in (\R^{d})^N\, :\, \nabla D_{ij}({\bf x}) \cdot {\bf u} \geq 0  \hbox{ for all $i,j$ such that } D_{ij}({\bf x}) = 2\delta\}.
$$
Alternatively, using the fact that $\frac{d}{dt} |x_i(t)-x_j(t)|^2 = 2 (x_i(t)-x_j(t))\cdot (u_i({\bf x}(t))-u_j({\bf x}(t)))$, we can write 
\begin{equation}\label{eq:C}
C({\bf x})= \{{\bf u}\in (\R^{d})^N\, :\, 
(u_i-u_j)\cdot (x_i-x_j) \geq 0  \hbox{ for all $i,j$ such that } |x_i-x_j| = 2\delta\}.
\end{equation}
Following \cite{Maury11,MRSV}, one can show that the projection $ {\bf u}=P_{C({\bf x})} {\bf v}$ in $(\R^{d})^N$ can be characterized as follows:
Using the notation $i\sim j$ to indicate that the particle $i$ is in contact with the particle $j$ and $e_{ij}  = \frac{x_j-x_i}{|x_j-x_i|}$, we have
\begin{equation}\label{eq:uv}
u_i = v_i - \sum_{j\sim i} p_{ij} e_{ij}
\end{equation}
where the pressure $p_{ij}=p_{ji}$ satisfies
$$
p_{ij}\geq 0, \quad u_i \cdot e_{ij}+ u_j\cdot e_{ji}\leq 0, \qquad \forall i\sim j 
$$
(this last condition expresses the non-overlapping constraint)
and 
 $$
 u_i \cdot e_{ij}+ u_j\cdot e_{ji}=0 \mbox{ if } p_{ij}\neq 0.
 $$
This condition expresses the fact that if $p_{ij}\neq 0$, then the particles will stay in contact, i.e. the pressure prevents overlapping but does not push the particles away from each other.

\medskip

The existence of a unique solution to~\eqref{eq:ODEu} (for $\bf v$ Lipschitz and bounded) is proved in  \cite{Maury11} (see also \cite[Theorem 3.2]{MRS}). Except in the case of $d=1$, this does not fall into standard ODE theory because the set of admissible particle configurations 
\[
Q = \{\mathbf{x} = (x_1,x_2,\dots, x_N)\in(\R^{d})^N \, : \, |x_i - x_j|\ge 2\delta, \quad \forall i\neq j\}
\]
is not convex in general. But it satisfies the weaker property of \textit{unifom prox-regularity} which is sufficient for well-posedness.

\medskip

As a final comment, we recall that in this microscopic model, it is possible to consider that each particle $x_i$ has its own desired velocity $v_i$ (see \cite{MRSV}). This distinction is not possible in the macroscopic model, which does not distinguish particles from each other.  Thus in consideration of passing to the macroscopic model  it makes sense to assume that the velocity of the $i^{th}$ particles $v_i({\bf x})$ is equal to $\na \phi (x_i)$ for some potential $\phi:\R^d\to \R$.
Given the corresponding projection ${\bf u}$, we denote by $\mathbb{P}_C(\na \phi)$ a vector field $\R^d\to \R^d$ such that $\mathbb{P}_C(\na \phi)(x_i) = u_i({\bf x})$ (only the values of this vector field at the points $x_i$ are uniquely defined).
We can then write that the empirical distribution satisfies
\begin{equation} \label{eq:empc}
\pa_t \rho_N + \div(\rho_N \mathbb{P}_C(\na \phi))=0, \qquad \mu_N(t) = K_\delta *\rho_N \leq 1 .
\end{equation}
 \smallskip

\subsection{The macroscopic hard-sphere model}\label{sec:hs}
Formally, the same principle can be applied to the macroscopic system:
Here, we consider the formulation
 \eqref{eq:empc} of  \eqref{eq:ODEu} and we write that the density $\rho(t,x)$ solves the continuity equation
$$\pa_t \rho + \div (\rho\,  \mathbb{P}_{C_\rho}(\na \phi))=0, \qquad \rho\leq 1.$$
The set of admissible velocities, which preserves the local constraint $\rho\leq 1$, can now be defined by duality as
\[
C_\rho = \left\{
U \in L^2(\R^d; \, \R^d) \, ;\, \int_{\R^d} U\cdot \na q \le 0, \quad \forall q\in H^1_\rho(\R^d)
\right\},
\] 
where
\[
H^1_\rho(\R^d) = \left\{
q\in H^1(\R^d) \, ;\, q\ge 0\, \text{ a.e. in }\R^d \text{ and }q(x)(1-\rho(x)) = 0 \text{ a.e.}
\right\}
\]
(we can also write $H^1_\rho(\R^d) = \left\{q\in H^1(\R^d) \, ;\, q(x)\in \pa f_\infty(\rho(x)) \text{ a.e.}\right\}$, but the condition $q(x)(1-\rho(x)) = 0 $ is a common way of writing the condition $q(x)=0 \, \text{ a.e. in }\{\rho <1\} $)

Formally at least, this projection operator guarantees that the density satisfies the incompressibility constraint  $\rho\leq 1$
since it implies $\div U\geq 0$ in (the interior of) the set $\{\rho=1\}$.
In order to develop a well posedness theory for this model, we first show that it can be rewritten without the projection operator via the introduction of a pressure term (as is commonly done in incompressible fluid dynamic - although in our case the constraint is $\rho\leq 1$ and not $\rho=1$).

To derive this formulation, we note that the projection $U:=\mathbb P_{C_\rho} (\na \phi)$ is  characterized by the following inequality:
\begin{equation}\label{eq:b}
\int_{\R^d} (U-\na \phi)\cdot (U-Q) \, dx \leq 0, \qquad\forall Q \in C_\rho.
\end{equation}
This projection  gives rise to a Lagrange multiplier in the form of a pressure term: One can show\footnote{
This follows from Moreau's decomposition theorem since the set $\mathcal K_\rho = \{\na p\,;\, p\in  H^1_\rho(\R^d)\}$
is a closed convex cone in $L^2(\R^d; \, \R^d)$ and its polar cone is exactly $\mathcal K_\rho^\circ = C_\rho$.
It follows that all $\na \phi\in L^2(\R^d; \, \R^d)$ can be written in a unique way as
$ \na \phi = \mathbb{P}_{ C_\rho} (\na \phi) +  \mathbb{P}_{\mathcal K_\rho} (\na \phi).$
}
 that there must exist $p\in H^1_\rho(\R^d)$ such that 
$$U= \na \phi - \na p.$$
The velocity $U =  \mathbb{P}_{C_\rho} (\na \phi)$ is in $C_\rho$ and thus satisfies
$\int_{\R^d} U \cdot \na q \, dx \leq 0$ for all $ q\in H^1_\rho(\Omega)$ and the pressure satisfies the orthogonality condition (which follows from \eqref{eq:b} by taking $Q=0$ and $Q=2U$):
$
 \int_{\R^d} U \cdot \na p \, dx = 0.
$
We deduce that $-\int_{\R^d} U\cdot \na (p-q)\, dx \leq 0$ for all $q\in H^1_\rho$, which  implies that
$p$ solves the variational inequality
\begin{equation}\label{eq:obstacle}
\begin{cases}
p\in H^1_\rho \\
\displaystyle \int_{\R^d}[ \na p-\na \phi] \cdot \na (p-q) \leq 0 \, , \qquad \forall q\in H^1_\rho
\end{cases}
\end{equation}
This variational inequality is in fact an obstacle problem (because of the constraint $p\geq 0$).
We can summarize this discussion with the following proposition (see \cite[Appendix B]{KMW}):
\begin{proposition}
Assume $\phi \in H^1$ and 
let $p(x)$ be the (unique) solution of the variational inequality  \eqref{eq:obstacle}. Then 
$ \mathbb{P}_{C_\rho} (\na \phi) = \na \phi - \na p $.
\end{proposition}
The macroscopic hard-sphere model can thus be written as:
\begin{equation}\label{eq:incomp2}
 \pa_t \rho + \div(\rho \na \phi) = \div (\rho \na p), \qquad p\in \pa f_\infty(\rho).
 \end{equation}

\subsection{Well-posedness of the PKS hard-sphere model}
In this section, we discuss the well-posedness of the model \eqref{eq:incomp2} when the velocity field $V$ is the attractive  chemotaxis gradient from \eqref{eq:LHS}. We can also write this as:
\begin{equation}\label{eq:PKSincomp}
 \pa_t \rho + \div(\rho \na G* \rho) = \div (\rho \na p), \qquad p\in \pa f _\infty(\rho).\\
 \end{equation}
A solution of  \eqref{eq:PKSincomp} can be constructed either as a gradient flow or as the singular limit of the  nonlinear diffusion model.

The gradient flow approach was first developed in the context of congested crowd motion in~\cite{MRS}. In that paper, the velocity field was given by $V=-\na \phi$ for a continuous and $\lambda$-convex potential $\phi$. 
A similar approach was later carried out with the chemotaxis potential in \cite{KMW}:
Equation \eqref{eq:PKSincomp} is the 2-Wasserstein gradient flow of the energy $\E_\infty$ defined by \eqref{eq:Einfty} 
and existence of a solution can be proved via a discrete time approximation based on the classical JKO minimization scheme. In fact, we have:
\begin{theorem}\label{thm:existenceinc}
Let $G$ be given by \eqref{eq:G}. Given $\rho_{in}\in \mathcal P_2(\R^d)$ such that $\rho_{in}\leq 1$ and 
$\supp \rho_{in} \in B_R$.
Then Equation \eqref{eq:LHS} (or \eqref{eq:PKSincomp}) has a unique global solution $\rho(t,x)$ with initial data $\rho_{in}(x)$.
Furthermore, there exists $R(t)$ such that
$$  \supp \rho(t) \subset B_{R(T)} \quad \forall t\in[0,T].
$$ 
The density satisfies $\rho\in C^{1/2}([0,\infty);\P_2(\R^d))$ and the pressure is such that $p\in L^2(0,\infty;H^1(\R^d))$, which is enough to make sense of the condition $p(1-\rho)=0$.
\end{theorem}

When the drift term $\na \phi$ is fixed, uniqueness with density constraint is proved in  \cite{DiMe16} assuming some monotonicity for $\na \phi$, and in \cite{Igbida} in the spirit of DiPerna-Lions theory for transport equations, assuming some sobolev regularity for $\phi$.
Uniqueness for a related model (without the drift term) was also proved in \cite{PQV} using a duality method that turns out to be very adaptable and was used in \cite{KMW} and \cite{HLP} to prove uniqueness for \eqref{eq:PKSincomp}.

Theorem \ref{thm:existenceinc} (existence and uniqueness) also follows from \cite{HLP}, where a solution is constructed by passing to the limit $m\to \infty$ in the soft sphere model \eqref{eq:LSS}. 
Note that the   incompressible limit $m\to\infty$ for \eqref{eq:LSS} is a classical problem that goes back to the 80's: the ``stiff (or incompressible) limit" of the porous media equation.
We can understand this limit by noticing that \eqref{eq:LSS}  has the same form as \eqref{eq:PKSincomp} but with $p=f'_m(\rho)$.
We can also make sense of this limit at the level of the energy. Indeed,
equation \eqref{eq:LSS} is the 2-Wasserstein gradient flow 
for the energy $\E_m$ defined by \eqref{eq:Ef}
which $\Gamma$-converges to $\E _\infty[\rho]$. The $\Gamma$-convergence of the energy does not imply the convergence of the corresponding gradient flow, but it is a good indication that such a convergence might take place.

There are numerous references for this limit $m\to\infty$ in various frameworks: We refer to \cite{Vaz2007,GQ03,EHKO1986}  for the classical porous media equation $\pa_t \rho= \div(\rho\na  f_m'(\rho))$ and to  \cite{PQV,DP,KPS,MPQ,KP, KT,DS} for more recent results in the context of tumor growth modeling, which includes a non-negative source term $\pa_t \rho = \div(\rho\na  f_m'(\rho)) + \rho F(f_m'(\rho))$ (more recently \cite{GKM},  these results were extended to the case of the porous media with an absorption  term, which leads to interesting behavior).
For the equation with a fixed drift term, the limit was first investigated with a fixed potential
$\pa_t \rho + \div(\rho\na \phi ) = \div(\rho\na  f_m'(\rho)) $ which is an important pedestrian model for congested crowd motion (see \cite{CKY,AKY,BKP,DDP23}). 
Finally, the limit was established for the attractive chemotaxis problem \eqref{eq:G} with pure Newtonian attraction $\sigma=0$ (the analysis can be extended to the case $\sigma>0$, which does not affect the singularity of the kernel at $0$):
This was first done in \cite{CKY} for particular solutions (patch solutions) 
using variational methods and the viscosity solution approach. The convergence of the gradient flow of $\E _m$ for general initial conditions was then proved in \cite{CT20}, and finally, the convergence of weak solutions was proved in \cite{HLP} by a more direct approach (which can be generalized to include source terms for example, see  \cite{JKT21}).

\medskip

The key feature of this limit is the fact that $p_m(\rho)=f'_m(\rho) \to \pa f_\infty(\rho)$ when $m\to\infty$.  There are other natural pressure terms that lead to a similar result.
For example, one can replace the pressure $p_m(\rho)$   with a singular pressure such as 
\begin{equation}\label{eq:pa1}
p_\alpha (\rho)=  
\begin{cases}
 \alpha \frac{\rho}{1-\rho} & \mbox{ if } \rho<1\\
 \infty & \mbox{ if } \rho\geq 1.
 \end{cases}
\end{equation}
or 
\begin{equation}\label{eq:pa2}
p_\alpha (\rho)=
\begin{cases}
-\alpha  \log(1-\rho)& \mbox{ if } \rho<1\\
 \infty & \mbox{ if } \rho\geq 1.
 \end{cases}
\end{equation}
Such pressure terms lead to solutions satisfying $\rho<1$ (for $\alpha>0$), but in the limit $\alpha \to 0$, we find $p_\alpha (\rho)\to \pa f_\infty(\rho)$.
The corresponding limits were made rigorous in   \cite{HV17,DHV,DDP23}  for \eqref{eq:pa1} (with a fixed potential $\phi$ and with convergence rates) and in  \cite{Tadmor} for \eqref{eq:pa2} (with the attractive chemotaxis potential with $\sigma=0$).

\subsection{The projected velocity}\label{sec:proj}
In the approach discussed above, solutions of 
 \eqref{eq:PKSincomp} are constructed by imposing the constraint $\rho\leq 1$, rather than by using the projection $\mathbb P_{C_\rho}(\na \phi)$. 
However, it is possible to show that the effective velocity in  \eqref{eq:PKSincomp}, $\na \phi -\na p$, is indeed the expected projection.
Following \cite{MRS}, one can show that
$$
\int_{\R^d} \na q(x) \cdot (\na p(t_0,x) - \na \phi(t_0,x)) \, dx = 0 \qquad\forall q\in H^1_{\rho(t_0)}, \mbox{ a.e. } t_0>0.
$$
Since $p\in H^1_{\rho(t_0)}$, this implies that \eqref{eq:obstacle} holds for a.e. $t_0>0$. In fact it implies the   stronger statement 
$$ 
\int_\Omega[ \na p-\na \phi] \cdot \na (p-q) = 0\qquad \forall q\in H^1_{\rho}.
$$
By taking $q(x)=p(t_0,x) \vphi(t_0,x))$, we deduce in particular
$$
\int_{\R^d} \na (p(t_0,x)\vphi(t_0,x))  \cdot (\na p(t_0,x) - \na \phi(t_0,x)) \, dx = 0 \qquad\mbox{ a.e. } t_0>0,
$$
which implies  the so-called {\it complementarity condition}
\begin{equation}\label{eq:comp}
 p (\Delta p -\Delta \phi)=0 , \mbox{ in } \mathcal D'((0,T)\times \R^d), \quad p(t)\in H^1_{\rho(t)}.
 \end{equation}

It is also possible to derive these properties as part of  the construction process.
Indeed, both the solution of the nonlinear diffusion equation and the discrete time approximation given by the JKO scheme satisfy approximation of  \eqref{eq:comp}.
Intuitively, passing to the limit in these approximations require the strong convergence of the pressure in $H^1(\Omega)$.
This is done in \cite{PQV} in the time-monotone  model for tumor growth. The case of the more complex model with a term representing nutrients (which is no longer monotone), was then treated in \cite{DP} by generalizing the so-called 
 Aronson-B\'enilan estimates to get some strong convergence on the gradient of the pressure.

This derivation was also performed in \cite{GKM} by showing directly that  for all $t>0$, the pressure $p(t)$ is the unique solution of the variational inequality \eqref{eq:obstacle}. 
This approach has two advantages: This formulation of the complementarity equation can be derived without requiring any strong convergence on $\na p$, and unlike \eqref{eq:comp}, it identifies the pressure for all time (rather than a.e.). In problems without  monotonicity in time, the support of the pressure can indeed be smaller than the set $\{\rho=1\}$, an important phenomenon which is clearly identified with this obstacle problem formulation.
This obstacle problem \eqref{eq:obstacle}  was derived for the limit of the JKO solution in \cite{KMW} as well.

\subsection{Relation between the microscopic and macroscopic models}\label{sec:mm}
Since both the microscopic particles model \eqref{eq:ODEu} and macroscopic model \eqref{eq:incomp2} have solutions, it is natural to ask whether these solutions converge to one another when $\delta\to0$ (for the purpose of this discussion, we focus on the case of a fixed velocity field $V(x)=\na \phi(x)$).
However, unlike the soft sphere model discussed in Section \ref{sec:blob}, the relation between the microscopic and macroscopic models is far from clear in the hard-sphere case. 
We recall that in the soft sphere case, this relationship was addressed in two steps: First there was a one-to-one correspondence between the solutions of the ODE system \eqref{eq:PKSmicro} and the empirical measures solutions of the PDE \eqref{eq:NLSS}. Second, solutions of  \eqref{eq:NLSS} (for absolutely continuous initial data) were proved to converge to solutions of \eqref{eq:LSS} as $\delta\to0$.
We will show below that the first step still holds in the hard-sphere models: Solutions of the system of ODE with hard-sphere constraint \eqref{eq:constraint} are associated to empirical measures solutions of \eqref{eq:NLHS}. 
But the second step in the limiting process ($\delta\to0$) is delicate at best. In fact, it is not obvious that the equation \eqref{eq:NLHS} is well-posed outside of the particular case of empirical measures associated to \eqref{eq:PKSmicro}. Without a well-posedness theory, a convergence result (the equivalent of Theorem \ref{thm:convdelta}) seems out of reach.

Beyond the technical difficulties, we should also point out that the dynamics of the microscopic and macroscopic models are quite different. For example, when the particle $i$ is in contact with several other particles, which is expected in congested configurations, the microscopic condition requires control on the velocity $v_i$ in every direction $x_i-x_j$ for which $D_{i,j}=0$. By contrast, the macroscopic constraint is a unique scalar constraint on the divergence of the velocity field. 
The dynamic of the microscopic model is thus much richer than that of the macroscopic model.
Simple examples (see \cite{M18}) show that slightly different arrangements of the particles at the microscopic level (corresponding to the same macroscopic distribution) can lead to completely different behavior.
\medskip

We now show that the nonlocal PDE \eqref{eq:NLHS} corresponds to the microscopic hard-sphere model when $K_\delta = \frac{1}{\omega_d \delta^d}\chi_{B_\delta}$.
We already noted (see \eqref{eq:density_nonlocal}) that the non overlapping constraint can be written as $\mu_N=K_\delta * \rho_N\leq 1$.
Next, we describe the set of admissible velocity: if $\pa_t \rho_N + \div (\rho_N U)=0$, and $ p\geq 0$ is a measure  supported on the set $\{\mu_N=1\}$, then the condition ``$\pa_t \mu_N\leq 0$ when $\mu_N=1$" implies
$$
 \frac{d}{dt}\int K_\delta * \rho_N p\, dx  =  \int  \rho_N\,  U \cdot  \na K_\delta* p \, dx \leq 0.
$$
We can thus define the set of admissible velocities as 
\begin{equation}
	\label{eq:barC}
\widetilde C_\delta(\rho_N) = \left\{ U:\R^d\to\R^d\,;\,  \int \rho_N \, U \cdot \na K_\delta* p\, dx \leq 0, \mbox{ for all } p\geq 0 , \quad p(1-K_\delta*\rho_N) =0\right\}.
\end{equation}
Importantly, this description at the level of empirical measures is consistent with \eqref{eq:C}: 
if $U(x_i) = u_i({\bf x})$, then the condition ${\bf u} \in C_\delta({\bf x})$ from \eqref{eq:C} is equivalent to $U\in \widetilde C_\delta(\rho_N) $.
For example, if the two particles $i$ and $j$ are touching (up to a translation, we can take  $x_i=-\delta e$ and $x_j=+\delta e$), then taking  $p(x) = (\delta - |x\cdot e|)_+\mathcal{H}^1|_{\R e}$ in \eqref{eq:barC} will lead to the condition in \eqref{eq:C}. 
\begin{figure}[H]
	\centering
	\begin{tikzpicture}
		\draw (2.5,2) arc (0:360:1);
		\node at (1.5, 2.3) {$x_i = -\delta e$};
		\filldraw[black] (1.5, 2) circle (2pt);
		\draw (2.5,2) arc (180:540:1);
		\node at (3.5, 2.3) {$x_j = \delta e$};
		\filldraw[black] (3.5, 2) circle (2pt);
		\draw[blue,dashed] (0,2) -- (5,2) node[pos=0.9, anchor = south west] {$p(x)$};
		\filldraw[blue] (2.5, 2) circle (2pt);
	\end{tikzpicture}
	\caption{Schematic of two particles with radius $\delta$ in contact along the axis $e$. The pressure $p(x) = (\delta - |x\cdot e|)_+ \mathcal H^{1}|_{\R e}$ can be highly singular since it only needs to activate along the axis $e$.}
	\label{fig:2_touch}
\end{figure}
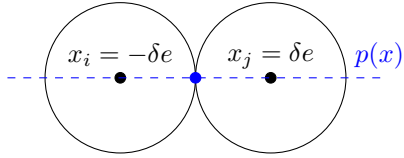
The simple illustration in Figure~\ref{fig:2_touch} sheds some light on another discrepancy between the microscopic model~\eqref{eq:ODEu} and the macroscopic model~\eqref{eq:incomp2} (already discussed in~\cite{MRSV,MRS}): the set of feasible velocities $C_\delta$ depends not only on the local density but also on the orientation of particles. In the macroscopic model~\eqref{eq:incomp2}, the feasible velocity only depends on the local density.

The set $\widetilde C_\delta(\rho_N)$ is the polar cone of
$$
\widetilde {\mathcal K}_\delta(\rho_N) = \left\{   \na K_\delta* p \, ;\,    p\geq 0 , \quad p(1-K_\delta*\rho_N) =0 \right\}
$$
in the Hilbert space $L^2(\R^d,d\rho_N)^d$.
In keeping with Moreau's decomposition theorem, we expect to be able to write 
$$
U =\mathbb P_{\widetilde C_\delta(\rho_N) }(\na \phi)= \na \phi - \na K_\delta*p
$$
for some pressure $p\geq 0$, satisfying $ p(1-\mu_N) =0 $. This is in fact correct when $\rho_N$ is an empirical distribution: 
This decomposition is equivalent to the (finite dimensional) projection of ${\bf v}$ given by \eqref{eq:uv} (see \cite{KMJW}).
The existence of a solution to the microscopic model (see Section \ref{sec:micro}) thus implies the existence of a solution (in the form of the associated  empirical measure) of
\begin{equation}\label{eq:incompdelta}
\pa_t \rho + \div(\rho \na \phi) = \div (\rho \na K_\delta * p), \qquad p\in \pa f_\infty(K_\delta*\rho) 
\end{equation}
when $\rho_{in} = \frac 1 N \sum_{i=1}^N \delta(x-x_i(0))$
(where we recall  that $ p\in \pa f_\infty(K_\delta*\rho) $ is equivalent to 
$K_\delta*\rho\leq 1, \; p\geq 0, \; p(1-K_\delta*\rho)=0$).

It is now natural to try to mirror the analysis described in Section \ref{sec:blob} with this model  for general density, namely to prove a mean field limit, as in Theorem 
\ref{thm:stab} (the constant $\lambda_0$ given by \eqref{eq:lambda0}   depends on $m$, so the extension to the case $m=\infty$ is not obvious) and the limit $\delta\to0$. 
There is however a serious obstruction: Except for the particular case of empirical measures (for which the existence of a solution follows from the existence of a solution for the microscopic model), the existence of solution for \eqref{eq:incompdelta} is far from clear.
A discrete time approximation can be constructed via a JKO scheme with the energy
$$
  -\int_\Omega \phi(x) \rho (x) \, dx + \mathcal F_{\infty}[K_\delta*\rho], 
$$
but proving that the resulting velocity is indeed the projection of $\na \phi$ onto $\widetilde C_\delta(\rho)$ is challenging. 
Even then, it should also be noted that for general density $\rho$, it is not clear that the cone $\widetilde{ \mathcal K}_\delta(\rho)$ is closed and thus Moreau's decomposition theorem does not apply. 
Similar difficulties arise when trying to obtain the existence of a solution  of \eqref{eq:NLHS} by passing to the limit $m\to\infty$ in the soft sphere model \eqref{eq:NLSS}. This limit has been studied for the related Brinkman law model \eqref{eq:brink} \cite{PV15,KT}, but this case is somewhat easier since the condition $\rho\leq 1$ (as opposed to $\rho*K\leq 1$) is being enforced.

\medskip

\section{Hard-sphere models and Hele-Shaw  Free boundary problems}\label{sec:hsFBP}
In some settings, the hard-sphere model \eqref{eq:LHS} might lead to the formation and propagation of sharp edges. In fact the connection between the hard-sphere model \eqref{eq:LHS} and Hele-Shaw free boundary problem (with active potential, but without surface tension) is well documented:
If we assume that $\rho (t,x)= \chi_{E(t)}$ and denote $\phi(t,x) = G* \chi_{E(t)}(x)$, then \eqref{eq:LHS} implies (formally at least):
\begin{equation}\label{eq:HS1}
\Delta (p-\phi) = 0\quad  \mbox{ in } E(t),\qquad V = -\na (p-\phi)\cdot \nu\quad \mbox{ on } \pa E(t),
\end{equation}
where $V$ denotes the normal velocity of the interface $\pa E(t)$ and $\nu$ is the outward unit normal vector along $\pa E(t)$,
with the condition
\begin{equation}\label{eq:HS2} 
p(t,x) = 0\quad \mbox{ on } \pa E(t) .
\end{equation}
Equations \eqref{eq:HS1}-\eqref{eq:HS2} (which one can write with the effective pressure $q=p-\phi$) is reminiscent of a Hele-Shaw free boundary problem  (without surface tension) or one-phase Muskat problem.

\medskip

However, the assumption $\rho(t,x)=\chi_{E(t)}$ is not typically satisfied by all solutions of  \eqref{eq:LHS}. Since the equation only imposes the constraint $\rho(t,x)\leq 1$, there are no obvious reasons why $\rho$ should be a characteristic function.
But it can sometimes be proved that phase separation is propagated by the equation: If $\rho_{in}=\chi_{E_{in}}$, then $\rho(t)=\chi_{E(t)}$ for all $t>0$.
Such a result was proved in \cite{MPQ} and \cite{KP} for a different model, involving a non-negative source term but no drift. 
The crucial ingredient in these papers was the monotonicity of $\rho$ with respect to $t$, and it was proved that $\rho(t) = \chi_{E(t)}$ a.e. $t>0$ for some open set $E(t)$ with finite perimeter. 
For the aggregation equation, a similar result was proved in \cite[Theorem 1.1 (b)]{CKY} when $G$ is the purely Newtonian attraction kernel using viscosity solution approach.
Similar results were proved in \cite{AKY}  when $\phi$ is  a fixed potential satisfying $-\Delta \phi \geq 0$ (see also \cite{BKP} where the assumption is on the source term).
This assumption, which implies some monotonicity of $\rho$ along characteristic curves, is essential in these works.
When $G$ is given by \eqref{eq:G}, the condition $-\Delta \phi \geq 0$ is no longer satisfied but the fact that this potential is always  attractive suggests that there is still some intrinsic monotonicity in the model.
Indeed, the propagation of phase separation was proved  in  \cite{KMW} for the hard-sphere PKS model (that is \eqref{eq:LHS} with $G$ given by \eqref{eq:G}). The proof is very different from those mentioned above and the result is somewhat weaker:
 it is proved that for appropriate initial data and for a.e. $t>0$ there is a set $E(t)$ such that $\rho(t)=1$ a.e. in $E(t)$ and $\rho(t)=0$ a.e. in $\R^d \setminus \overline {E(t)}$. 
The condition $\rho = \chi_{E(t)}$ would follow if one could show that $|\pa E(t)|  =0$, a property that appears difficult to get.

\medskip

Making rigorous the connection between \eqref{eq:LHS} and the common notions of weak solutions for the Hele-Shaw free boundary problem (e.g. viscosity solutions) is a delicate issue.
Besides the question of phase separation, there is the issue of giving a meaning to \eqref{eq:HS1} when the set $E(t)$ is not smooth.
As discussed in Section \ref{sec:proj}, the first condition in \eqref{eq:HS1} can be derived rigorously (it is the complementarity condition).
The fact that the Hele-Shaw velocity law is satisfied in some viscosity solutions sense was proved in \cite[Theorem 1.1]{CKY} for the Newtonian kernel.

\medskip

\part{\Large Sharp interface limits and free boundary problems}
In the second part of this paper, we review and extend recent results showing that when $m>2$ and at appropriate  scales, the competition between  local repulsion and non-local attraction as described by equation \eqref{eq:LSS}
leads to phase separation and to the formation of an interface separating regions of high and low (or $0$) density, interface whose evolution can be described by  models of geometric nature.
As noted in the introduction, phase separation plays a critical role in many biological processes and it has been shown that it  can be explained by the simple attraction-repulsion mechanisms modeled by equation \eqref{eq:AD}.  We refer for instance to \cite{TBL,BT16} (for biological aggregations such as insect swarms) and to \cite{FBC} (where phase separation is discussed for a continuum model of cell-cell adhesion).
From a mathematical view point, the main idea behind this phenomena is the close relationship between the aggregation-diffusion  equation \eqref{eq:AD} and a degenerate Cahn-Hilliard equation.
In this part of the paper, we describe an approach, based on the recent work \cite{KMW2,M23},  to rigorously investigate phase separation phenomena and derive effective free boundary  problems by considering appropriate singular limits of~\eqref{eq:AD}.

\section{A non-local Cahn-Hilliard equation}\label{sec:FBP}
Our goal is to derive effective models for aggregation-diffusion phenomena  when
the total mass  $\int \rho(t,x)\, dx$ is very large and the phenomena are observed from far away (at an appropriate time scale): We introduce $\eps\ll1$ such that $\int_{\R^d} \rho(t,x)\, dx= \eps^{-d}$ and rescale the space and time  variables as follows:
 $$ x=\eps \bar x, \qquad t=\eps^{2} \bar t .$$
The function $\rho^\eps(t,x):=  \rho(\bar t, \bar x)$ now satisfies $\int_{\R^d} \rho^\eps(t,x)\, dx=1$ and solves \eqref{eq:LSSeps} which we recall here for the reader's convenience:
$$\pa_t\rho + \div(\rho \na (G_\eps*\rho)) = \div(\rho\na f'(\rho)), \qquad G_\eps(x) = \eps^{-d}G(\eps^{-1}x)$$
(alternatively, this amounts to studying  \eqref{eq:LSS} when the range of the nonlocal attraction  is small compared to typical macroscopic length).
Throughout this section, we will assume that $f$ satisfies
\begin{equation}\label{eq:ff} f(\rho)=f_m(\rho), \quad m>2, \quad\mbox{ or } \quad f(\rho)=f_\infty(\rho).\end{equation}
This assumption ensures that the function $h(\rho)$ defined by \eqref{eq:hh} is a double-well potential (see \eqref{eq:hm} and \eqref{eq:hinfty}), which is all we really need for the results presented here to hold.

Most of the rigorous results in this direction have only been proved when $G$ is the  chemotaxis potential \eqref{eq:G} or some other very particular kernel (such as the heat kernel). 
Extending the  theory (in  particular Theorem \ref{thm:HS}) to  general interaction potentials $G$ describing local attraction is an interesting and  important open problem. For the formal discussion below, it is enough to assume that
$$
 G=G(|x|)\geq 0, \qquad   \int G(x)\, dx= \sigma^{-1} \in(0,\infty) \quad \mbox{ and  } \quad \int_{\R^d} |x|^2 G(x)\, dx <\infty.
$$
 \medskip

 A crucial observation mentioned in the introduction is the fact that  \eqref{eq:LSSeps} can be written as a non-local Cahn-Hilliard equation (with a singular double-well potential).
Indeed, since $G_\eps * \rho \to \sigma^{-1} \rho $ when $\eps\to 0$, it makes sense to rewrite \eqref{eq:LSSeps} as 
\begin{equation}\label{eq:bfkhsd}
\pa_t \rho  = \div(\rho\na [f'(\rho)-  \sigma^{-1}\rho] )- \div(\rho \na [G_\eps *\rho-\sigma^{-1} \rho]) .
\end{equation}
Introducing the symmetric nonlocal operator 
$$
B_\eps[\rho] (x):= (G_\eps *\rho-\sigma^{-1} \rho)(x) = \int_{\R^d } G_\eps(x-y) [\rho(y)-\rho(x)]\, dy,
$$
and using the function $h$ defined by \eqref{eq:hh}, we can write \eqref{eq:bfkhsd} as
\begin{equation}\label{eq:NCH}
\pa_t \rho  + \div (\rho v)=0 ,\qquad 
v=- \na h'(\rho)+ \na B_\eps[\rho].
\end{equation}
This equation is the gradient flow for the energy $\J_\eps$ defined in \eqref{eq:Jeps} (which can be obtained from the energy \eqref{eq:Jeps0} by a similar computation) with respect to the 2-Wasserstein distance. Equation \eqref{eq:NCH} is closely related to a model derived in \cite{lebo97} by Giacomin and Lebowitz as a macroscopic model for phase separation in particle systems with long range interactions. This equation can be seen as a non-local approximation of the degenerate Cahn-Hilliard equation \cite{CH58,EG96}: When $\rho$ is a smooth function, we have (using the symmetry of $G$): 
$$
B_\eps[\rho] = \eps^2 \beta \Delta \rho + \mathcal O(\eps^4), \qquad \beta:=\int_{\R^d} (z\cdot e)^2 G(z)\, dz
$$
(for any $e\in \mathbb S^{d-1}$) and so \eqref{eq:NCH} is close to
\begin{equation}\label{eq:CH}
\pa_t \rho  + \div (\rho v)=0 ,\qquad 
v=- \na h'(\rho)+ \eps^2 \beta \na \Delta \rho.
\end{equation}
This connection between aggregation-diffusion equations and degenerate Cahn-Hilliard equations is discussed and used for instance in \cite{BT16,CCY19,FBC,KMW2,M23}.
When $h$ is a double-well potential (as is the case when \eqref{eq:ff} holds), it is well-known that \eqref{eq:CH} leads to phase-separation and can be approximated - at appropriate time scales - by Stefan or Hele-Shaw free boundary problems.
The sharp interface limit $\eps\to0$ was first investigated via formal expansion method by Pego \cite{Pego} (for the non-degenerate equation) and by Glasner \cite{Glasner} (for the degenerate equation). 
Rigorous justification of these formal asymptotic were established by Alikakos, Bates and Chen \cite{ABC,Chen96} for the non-degenerate equation.
Formal results were also obtained for the nonlocal equation \eqref{eq:NCH} by Giacomin and Lebowitz in \cite{lebo98}. 
In this part, we present some rigorous results for such limits for the aggregation diffusion equation \eqref{eq:NCH}.

\medskip

As a final comment, we note that the convergence of the rescaled version of  \eqref{eq:NCH} 
$$
\pa_t \rho  +\div(\rho v) =0 \, , \qquad v= -\na h'(\rho) +\eps^{-2} \na  B_\eps[\rho]
$$
to the  degenerate Cahn-Hilliard equation:
$$
\pa_t \rho   +\div(\rho v) =0 \, , \qquad v= -\na  h'(\rho) + \beta \Delta \rho 
$$
was recently proved in \cite{elbar2022degenerate} on the torus and when $G_\eps$ is a compactly supported approximation of unity (see also \cite{carrillo2023degenerate} for a similar result for  systems of Cahn-Hilliard equations).

\medskip

\section{Phase separation: Stefan free boundary problems}
Going back to  \eqref{eq:NCH}, we now investigate 
the asymptotic behavior of the solutions as   $\eps\to0$.
Since   $B_\eps[\rho]$ vanishes as $\e$ vanishes, we expect the asymptotic dynamics to be close to that of the  forward-backward diffusion equation
\begin{equation}\label{eq:fb}
\pa_t \rho = \div(\rho\na {h}'(\rho))
\end{equation}
(recall that under \eqref{eq:ff}, $\rho\mapsto h(\rho)$ is not convex). 
Such equations are commonly associated with the modeling of phase separation mechanisms, but are ill-posed from a mathematical point of view (see \cite{Hollig}).
One would expect the limits of solutions of \eqref{eq:NCH} to inherit some additional properties, but characterizing these limits is quite challenging, even for the classical Cahn-Hilliard equation.
The simplest setting is when the initial data does not take values in the "mushy" region $(0,\theta)$: such initial data are usually referred to as "well-prepared".
In that case,  the limit $\eps\to0$ of \eqref{eq:NCH} leads to 
the equation 
\begin{equation}\label{eq:stefan}
\pa_t \rho = \div(\rho\na {h^{**}}'(\rho)),
\end{equation}
where $h^{**}$ denotes the convex hull of $h$. 
When $h=h_m$ in \eqref{eq:hm}, we have 
$$
h^{**}_m (\rho)= 
\begin{cases}
0 & \mbox{ for } \rho \in [0,\theta_m];\\
h_m(\rho) & \mbox{ for } \rho\geq \theta_m.
\end{cases}
$$
With this choice of $h$ \eqref{eq:stefan}  corresponds to a generalized Stefan free boundary problem, where $\rho$ is  the enthalpy variable.
This equation propagates phase separation in the sense that the   mushy region $\{\rho(t,x)\in (0,\theta)\}$ is non-increasing with respect to $t$ (see for example  \cite{gotz}). 
So when $|\{\rho(x,0)\in (0,\theta)\}|=0$, this property remains true for all $t>0$, and there exists a set $E(t)$ such that
$$ \rho >\theta \mbox{ in } E(t), \qquad \rho = 0 \mbox{ in } \R^d\setminus E(t).
$$
The evolution of the sharp interface $\pa E(t)$ is then described by the Stefan problem \eqref{eq:stefan} which is a weak formulation of \eqref{eq:stef}.

\medskip

When the initial data is not well-prepared, further characterization of  the limits for solutions of \eqref{eq:NCH} is required to understand how phase separation arises.
This has been studied for the non-degenerate Cahn-Hilliard equation.
First we note (see Figure \ref{fig:1}) that there exists $\underline \theta\in(0,\theta)$
such that $h''>0$ in $(\underline \theta,\infty)$ and $h''<0$ in $(0,\underline \theta)$.
When $\rho_{in}$ does not take values in the unstable region $(0,\underline \theta)$, Plotnikov \cite{plot99} characterized the limits in dimension $1$ and showed that the limiting problem may exhibit  hysteresis phenomena. 
In general the  problem is very unstable and numerical computations show that oscillations appear in the region where $\rho_{in}(x)\in (0,\underline \theta)$
(see  \cite{elliott,BFG06}). Passing to the limit  $\eps\to 0$ requires a  weak formulation of   \eqref{eq:fb} involving Young's measures (see for example Plotnikov \cite{plot97} and  Slemrod \cite{slem91}).
We will not review these results here, but we stress the fact that the dynamics of this phase separation process in the set $\{\rho_{in} \in(0,\theta) \}$ is very rich, even in dimension $1$
(see \cite{BFG06}). We are not aware of any similar analysis for the degenerate Cahn-Hilliard equation or for its non-local approximation \eqref{eq:NCH}.

\medskip

Even with well-prepared initial data, 
justifying the convergence to the Stefan problem \eqref{eq:stefan} is delicate. 
A similar limit has been studied for the degenerate Cahn-Hilliard equation \eqref{eq:CH}.
The convergence of \eqref{eq:CH} to \eqref{eq:stefan}, studied formally in \cite{Pego}, was established rigorously only in dimension $1$ by Delgadino  in \cite{Delgadino}, for  well-prepared initial data. Delgadino relied on the approach developed by Sandier and Serfaty \cite{serfaty} to prove the convergence of gradient flows
using the fact that the Allen-Cahn energy \eqref{eq:ACdf} $\Gamma$-converges to $ \int_{\R^d} h^{**}(\rho(x))\, dx$.
Convergence to a Stefan free boundary problem was also proved by a similar approach for the  1-dimensional non-degenerate Cahn-Hilliard equation in  \cite{Bellettini}.
But as far as we know no such results are known for our non-local equation \eqref{eq:NCH}
except for some formal results for a nonlocal equation similar to \eqref{eq:NCH} in \cite{lebo98}
(via formal expansion and matched asymptotic analysis).
Extending the rigorous analysis of \cite{Delgadino} to this equation is an interesting open problem.
We note that such a result  appears plausible at the level of the energy: indeed 
  a corresponding  $\Gamma$-convergence result can be obtained for the  energy $\J_\eps$ given by \eqref{eq:Jeps}.
\begin{proposition}\label{prop:Gamma1}
Assume that $G \in L^1(\R^d)$ and that $\int_{\R^d} |z|^2 G(z) \,dz <\infty$. Then
$ \J_\eps$ (defined by \eqref{eq:Jeps}) $\Gamma$-converges to $\J^*$ (defined by \eqref{eq:J*}). More precisely, we have
\item[(i)] For any sequence $\rho^\eps \in \mathcal P_2(\R^d)$ that converges weakly to $\rho \in \mathcal P_2(\R^d) $, we have
$$ \liminf_{\eps\to 0} \J_\eps(\rho^\eps) \geq \J^*(\rho)$$
\item[(ii)] For any $\rho  \in \mathcal P_2(\R^d) $, there exists a sequence $\rho^\eps \in \mathcal P_2(\R^d)$ such that $\rho_\eps$ converges weakly to $\rho $ and 
$$ \limsup_{\eps\to 0} \J_\eps(\rho^\eps) \leq \J^*(\rho).$$
\end{proposition}
 We provide a short proof of this result in Appendix~\ref{app:Gamma}.

\medskip

While we will not attempt to extend the  analysis of  \cite{Delgadino} to our  nonlocal equation \eqref{eq:NCH}, we can prove the following conditional result:
\begin{theorem}\label{thm:stefan}
Assume that \eqref{eq:ff} holds and that $G$ is given by \eqref{eq:G}.
Let $\rho_{in}^\eps$ be a sequence of initial data in $\mathcal P_2(\R^d)$ satisfying $\J_\eps(\rho_{in}^\eps) \leq C$ and let $\rho^\eps(t,x)$ be the corresponding solution of \eqref{eq:LSSeps} up to time $T>0$. 
There exists a subsequence $\eps_n$ such that $\rho^{\eps_n}(t)$ converges  narrowly (locally uniformly in $t$) to $\rho\in AC([0,T];\mathcal P_2(\R^d))$.
Furthermore if we have 
\begin{equation}\label{eq:EAS}
\lim \int_0^T \J_{\eps_n}(\rho^{\eps_n}(t))\, dt \to \int_0^T \J^* (\rho(t))\, dt,
\end{equation}
then $\rho$ solves
\eqref{eq:stefan} on $[0,T]\times\R^d$
and  $\rho^{\eps_n}$ converges strongly (in $L^2$) to $\rho$ in $\{ h^{**}(\rho)>0\}$.
\end{theorem}
 As mentioned in the introduction,  the convergence assumption \eqref{eq:EAS} offers a classical approach to (conditional) convergence results for free boundary problems (see also the next section). 
We recall that the $\Gamma$-convergence of $\J_\eps$ implies the liminf inequality, 
so  \eqref{eq:EAS} says that there is no loss of energy in the limit. 
Theorem \ref{thm:stefan} appears to be new and we provide a proof in Appendix \ref{app:Gamma}.
While we do not explicitly require the initial data to be well-prepared in this result, assumption \eqref{eq:EAS} is in fact likely stronger:
In \cite{Delgadino} it is proved (in dimension $1$ and for the Cahn-Hilliard equation) that condition \eqref{eq:EAS} is  satisfied when the convergence of the energy holds for the initial data, that is, $\J_\eps(\rho_{in}^\eps) \to \J^*(\rho_{in})$.

\section{Surface tension phenomena: Hele-Shaw flows}\label{sec:st}
When $t\to \infty$, solutions of the Stefan problem \eqref{eq:stefan} converge to functions of the form $\theta \chi_{E_\infty}$, which are steady states of the Stefan problem (for any set $E_\infty$). Going back to our aggregation-diffusion equation \eqref{eq:LSSeps} (or \eqref{eq:NCH}), this means that solutions experience phase separation.
However, these simple characteristic functions are not steady states of  \eqref{eq:LSSeps}  and will thus continue to evolve at a much slower time scale. 
In order to characterize this evolution,  we now take an initial condition of the form
$$\rho_{in}=\theta \chi_{E_{in}} $$
in \eqref{eq:NCH}
and rescale the time variable 
$t\mapsto \eps^{-1} t$, leading to equation \eqref{eq:LSSeps2}, which we rewrite as
\begin{equation}\label{eq:NCH2}
\eps \pa_t \rho  = \div(\rho\na h'(\rho) )-  \div(\rho \na B_\eps[\rho]).
\end{equation}
In this section, we describe how the limit $\eps\to0$ gives rise to surface tension (mean-curvature) effects. In the next section~\ref{sec:ca}, we will show that when the equation is set on a bounded domain this limit also yields contact angle conditions.
This is not surprising, since for the corresponding rescaled Cahn-Hilliard equation: 
$$ \eps \pa_t \rho  = \div(\rho\na [ h'(\rho) - \eps^2 \Delta \rho ])$$
it was formally shown by Glasner \cite{Glasner} (using formal expansion and matched asymptotics) that the limit $\eps \to 0$ leads to a one-phase Hele-Shaw free boundary problem with surface tension.
We can also make sense of this limit by looking at the associated energy functional: Since
$h(\rho_{in})=0$, \eqref{eq:Jeps} reduces to
\begin{align*}
\J_\eps [\rho_{in}] & =  \frac{\theta^2}{4} \iint_{\R^d\times \R^d }  G_\eps(x-y)[\chi_{E_{in}}(x)-\chi_{E_{in}}(y)]^2\, dx\,dy\\
& =   \frac{\theta^2}{4} \iint_{\R^d\times \R^d }  G_\eps(x-y)|\chi_{E_{in}}(x)-\chi_{E_{in}}(y)|\, dx\,dy \\
& = \eps \gamma_0 P(E_{in}) + o(1)
\end{align*}
for some $\gamma_0>0$ (see for instance  \cite{BBM,MW}) where $P(E)$ denotes the perimeter of $E$, defined by \eqref{eq:Per}.
We thus introduce  the rescaled energy $ \G_\eps[\rho] := \eps^{-1} \J_\eps[\rho]$, that is
\begin{equation}\label{eq:Geps}
\G_\eps[\rho] 
=
\begin{cases}
\displaystyle \frac 1\eps \int_{\R^d} h(\rho) \, dx +\frac{1}{4\eps} \iint_{\R^d\times \R^d }  G_\eps(x-y)[\rho(x)-\rho(y)]^2\, dx\,dy
& \ds \text{ if }  \rho\in\P_2(\R^d), \; d\rho \ll \, dx\\
\infty & \mbox{ otherwise.}
\end{cases}
\end{equation}
which is a nonlocal approximation of the rescaled Allen-Cahn functional 
$$ \frac 1\eps \int_{\R^d} h(\rho) \, dx + \frac{\eps}{2}\int_{\R^d} |\na \rho|^2\, dx$$
whose 
  $\Gamma$-convergence  toward the perimeter functional is a classical result of  Modica-Mortola 
 \cite{MM77} (see also Modica \cite{M87} and Sternberg \cite{S88}).
 A corresponding  result for $\G_\eps$ was proved for rather general kernels $G$ when $h$ is a smooth double-well functional in \cite{AB98}. We recall that in our framework $h$ is not a smooth potential at $0$ (since $h(\rho)=-\infty$ when $\rho<0$ and $h(\rho)=+\infty$ when $\rho>1$ in the hard-sphere case $m=\infty$). Nevertheless, the proof can be adapted to this case.
When $G$ is given by \eqref{eq:G}, we have the following result (see \cite{MW} for the case $m=+\infty$ and \cite{M23} for the case $m\in(2,\infty)$):
\begin{theorem}\label{thm:Gamma}
Assume that $G$ is given by \eqref{eq:G} and that $f$ satisfies \eqref{eq:ff}.
Then  the energy functional  $ \G_\eps$ defined by \eqref{eq:Geps} $\Gamma$-converges to $\G_0$ defined by
$$
\G_0(\rho)= \begin{cases}
\displaystyle \gamma \int_{\R^d} |\na \rho|   & \mbox{ if } \rho\in \P_2(\R^d) , \quad \rho \in \BV(\R^d;\{0,\theta\})  \\
\infty & \mbox{ otherwise}
\end{cases}
$$
for some constant  $\gamma$ defined by \eqref{eq:sigma} below.
More precisely, we have 
\item[(i)] {\bf liminf property}: For all sequences $\rho^\eps \in L^1(\Omega)$ such that $\rho^\eps \to \rho$ in $L^1(\Omega)$, we have
$$
\liminf_{\eps\to 0 }\G _\eps(\rho^\eps) \geq \G_0(\rho).
$$
\item[(ii)] {\bf limsup property}: For all $\rho \in L^1(\Omega)$, there exists a sequence $\rho^\eps\in L^1(\Omega)$ such that $\rho^\eps\to\rho$ in  $L^1(\Omega)$ and 
$$\limsup_{\eps\to 0}  \G_\eps(\rho^\eps) \leq \G_0(\rho).$$
\end{theorem}
The definition of the constant $\gamma$ appearing in the definition of $\G_0$ requires the introduction of the function
$$
 g(s) : = \inf_{\rho\geq 0}  \left[ h(\rho) +  \frac {1}{ 2\sigma} (\rho-\sigma s)^2\right].
$$
We note that $g(s)\geq 0$ for all $s\geq 0$ and $g(0)=g(\theta/\sigma) =0$ (in fact $g$ is a double-well potential).
The constant $\gamma$  is then given by
\footnote{When $h=h_\infty$, we have $\theta=1$ and $g(s)=\frac \sigma 2 \min\{s^2,(s-\sigma^{-1})^2\}$ which leads to $\gamma = \frac{1}{4\sigma^{3/2}}$.}
\begin{equation}\label{eq:sigma}
\gamma: =\frac 1 \theta \int_0^{\theta/\sigma } \sqrt{2 g(s)}\, ds.
\end{equation}

\medskip

A key tool in the proof of this theorem in \cite{MW,M23} is the following alternative formula for $\G_\eps$ (with $\phi_\eps = G_\eps*\rho$):
\begin{align}
\G_\eps(\rho) & = \frac 1 \eps \int_{\R^d} h(\rho)  + \frac 1 {2\sigma}  (\rho-\sigma \phi ^\eps)^2\, dx +\eps  \int_{\R^d} \frac 1  2 |\na \phi_\eps |^2\, dx \label{eq:Jepschemo} \\
& \geq \frac 1 \eps \int_{\R^d} g(\phi_\eps )   dx +\eps  \int_{\R^d} \frac 1  2 |\na \phi_\eps |^2\, dx \label{eq:Jepschemo2}.
\end{align}
Inequality \eqref{eq:Jepschemo2} shows a stronger connection between $\G_\eps$ and the Modica-Mortola functional with double-well potential $g$ (which explains the role of the function $g$ in determining the constant $\gamma$ in $\G_0$).
Formula \eqref{eq:Jepschemo} is the key to proving another important property of this energy: The convergence of the energy implies the convergence of the corresponding first variation.
Such results have a long history:
Reshetnyak \cite{Reshetnyak} proved that if  $\chi_{E_\eps}$ converges to $\chi_E$ strongly in $L^1$, then the convergence of the perimeter 
$ P(E_\eps)\to P(E)$
implies the convergence of its first variation
$$ \delta P(E_\eps)[\xi]:=\int ( \div \xi - \nu_\eps\otimes\nu_\eps : D\xi)\, |D\chi_{E_\eps}|  \longrightarrow \int(\div \xi - \nu \otimes\nu:D\xi) |D\chi_E| \qquad \forall \xi\in C^1_c(\R^d;\R^d)$$
where $\nu$ is the density $\frac {D\chi_E}{|D\chi_E|}$ (which exists by Radon-Nikodym's differentiation theorem) and can be interpreted as the normal vector to $\pa E$.
We recall that the first variation of $P$ is the mean-curvature. Indeed, for a smooth interface $\pa E$ we have
\begin{equation}\label{eq:mc}
 \int(\div \xi - \nu \otimes\nu:D\xi) |D\chi_E| = \int_{\pa E} \div \xi - \nu\otimes\nu : D\xi\,   d\H^{d-1}  = \int_{\pa E} \kappa\, \xi\cdot\nu\, d\H^{d-1} 
\end{equation}
where   $\kappa$ denotes the mean curvature of $\pa E$.
A similar result was proved by Luckhaus and Modica \cite{LM} for the Ginzburg-Landau functional 
$$ \int_\Omega  \frac 1 \eps (1-\rho^2)^2 + \eps |\na \rho|^2 \, dx.$$
But as far as we know,  corresponding results for the nonlocal  functional $\G_\eps$ are only known for particular choices of the  kernel $G$:
 Laux and Otto \cite{LO16} proved it when  $G$ is the Gaussian and used it in the study of the thresholding scheme for multi-phase mean-curvature flow
(see also \cite{EO15,EM17,EE18}). 
Jacobs, Kim and M\'esz\'aros \cite{JKM} proved a similar result and used it to derive the Muskat problem. It should be noted however that in these references, the convergence is only proved when $\G_\eps$ is restricted to characteristic functions.
In \cite{KMW2}, we proved the corresponding result when $G$ is given by \eqref{eq:G} in the hard-sphere case $m=+\infty$, and then in \cite{M23} for the same kernel in the soft-sphere case $m\in (2,\infty)$ without restriction to characteristic functions.
We state here the corresponding result in our framework (see \cite{M23} - see also \cite{KMW2} for the case $m=\infty$): 
\begin{proposition}\label{prop:firstvar}
Assume that $G$ is given by \eqref{eq:G} and that $f$ satisfies \eqref{eq:ff}.
Given any sequence $\rho^\eps\in L^1$ which strongly converges to $\rho$ in $L^1(\R^d)$, if 
$$
\lim_{\eps\to 0} \G_\eps (\rho^\eps) =\G_0 (\rho) <\infty,
$$
then for all $\xi \in C^1_c(\R^d , \R^d)$ we have:
$$
\lim_{\eps\to 0} -  \eps^{-1}   \int_{\R^d}
 (h'(\rho^\eps) - B_\eps[\rho^\eps]) \na \rho^\eps \cdot \xi\, dx  = \gamma \int_{\R^d} \left[  \div \xi - \nu \otimes \nu :D\xi\right] |\na \rho|
$$
where $\nu = \frac{\na \rho}{|\na \rho|}$.
\end{proposition}

This proposition is the cornerstone of the proof of the following theorem which makes clear the connection between the rescaled aggregation-diffusion equation \eqref{eq:NCH2} and the Hele-Shaw free boundary problem with surface tension \eqref{eq:HS}, where the latter is the gradient flow of the perimeter functional $\G_0$ with respect to the $2$-Wasserstein distance.
\begin{theorem}
\label{thm:HS}
Given a sequence of initial data $\rho_{in}^{\eps_n}$ such that $\G_{\eps_n}(\rho_{in}^{\eps_n})\leq M$ and $\rho_{in}^{\eps_n}\to \rho_{in} = \theta \chi_{E_{in}} \in BV(\R^d;\{0,\theta\})$,  let $\rho^{\eps_n}(t,x)$ be the weak solution 
of \eqref{eq:LSSeps2}  with initial data $\rho_{in}^{\eps_n}$. 
Then the following hold:
\item[(i)] Along a subsequence, the density $\rho^{\eps_n}(t,x)$  converges strongly in $L^\infty(0,T;L^1(\R^d))$ to 
$$\rho(t,x) = \theta \chi_{E(t)} \in L^\infty(0,\infty;\BV(\R^d;\{0,\theta\}))$$
and there exists a velocity function $v(t,x)$ such that 
\begin{equation}\label{eq:contweakHS}
\begin{cases}
\pa_t\rho + \div(\rho v)=0, \\
\rho(x,0) =\theta \chi_{E_{in}}.
\end{cases}
\end{equation}
 \item[(ii)] If the following energy convergence assumption holds:
$$
\lim_{n\to\infty} \int_0^T \G_{\eps_n}( \rho^{\eps_n} (t)) \, dt = \int_0^T \G_0(\rho(t))\, dt,
$$
then there exists $p\in L^2(0,T;(C^s(\R^d))^*)$ (for any $s>0$) such that
\begin{equation}\label{eq:weakp}
\int_0^\infty \int_{\R^d} (\rho v\cdot \xi - \theta p \, \div\xi)\, dx \, dt= -\gamma  \int_0^\infty \int_{\R^d} \left[\div \xi - \nu\otimes\nu:D\xi\right] |\na \rho|\, dt
\end{equation}
for any vector field  $\xi \in C^\infty_c((0,\infty)\times  \R^d ; \R^d)$.\footnote{ The integral $\int_{\R^d} p \,  \div \xi\, dx$ above should be understood as the duality bracket
$\langle q , \div \xi \rangle_{(C^s(\Omega))^*,C^s(\Omega)}$.}
Together, equations \eqref{eq:contweakHS} and \eqref{eq:weakp} say that  $E(t)$  is a weak solution of the Hele-Shaw free boundary problem with surface tension \eqref{eq:HS} with  initial condition $E_{in}$.
\end{theorem}
The notion of weak solution of \eqref{eq:HS} that we recover with this theorem is similar to the definition considered for example in \cite{JKM,KMW2,Laux1}:
The continuity equation \eqref{eq:contweakHS} encodes the incompressibility condition $\div v=0$ in $E(t)$ and the free boundary condition $V=-v\cdot \nu$.
By taking test functions $\xi$ supported in either $E(t)$  or $\R^d\setminus E(t)$, we see that
 \eqref{eq:weakp} implies $\na p=0$ ourtside of $E(t)$ and $  v = -\na p$ in $E(t)$. 
 For general test functions  $\xi$, and taking into account the right hand side of  \eqref{eq:weakp} we further get 
 the surface tension condition $[p]  =  \gamma \kappa $ on $\pa E(t)$
as a consequence of \eqref{eq:mc}. 
In the radial setting, \eqref{eq:mc} implies that $p$  equals the constant curvature $\pa E(t)$ in a strong sense, thus yielding discontinuity of the pressure across the interface.

\begin{proof}
This result is proved in  \cite{KMW2} when $m=\infty$ and \cite{M23} when $m\in(2,\infty)$ when the equation is set in a bounded domain.
The proof can be extended to $\R^d$ with the use of the estimate on the second moment $\int_{\R^d} |x|^2\rho\, dx$.
Indeed, the only required adjustment is in the proof of \cite[Proposition 4.2]{M23}, where the argument only yields the strong convergence of $\rho^{\eps_n}$ to $\rho$ in $L^\infty(0,T;L^1_{loc}(\R^d))$. 
But we can write
\begin{align*}
\limsup \int_{\R^d} |\rho^{\eps_n}(t)-\rho(t)|\, dx &  \leq \limsup \int_{B_R} |\rho^{\eps_n}(t)-\rho(t)|\, dx+\limsup \int_{\R^d\setminus B_R} |\rho^{\eps_n}(t)-\rho(t)|\, dx\\
&  \leq 0 + \frac{1}{R^2} \int_{\R^d} |x|^2(\rho^{\eps_n}(t)+\rho(t)) \, dx
\end{align*}
and the strong  convergence in $L^\infty(0,T;L^1(\R^d))$ follows by taking $R\to\infty$.
\end{proof}
\medskip

\noindent {\bf More general nonlinearity $f$ and interaction kernel $G$:}
Theorem \ref{thm:HS} holds for general convex nonlinearities $f(\rho)$, provided the function $h(\rho)$ is a double-well potential.
For technical reasons, the proof also requires that  $f(\rho)\geq \left( \frac 1 {2\sigma}+\eta\right)\rho^2$ for large $\rho$.
In particular, adding a small linear diffusion term to the equation, which adds the term $\nu \rho\log\rho$ to the function  $h(\rho)$, will  displace the well at $\rho=0$ to some (small) positive value.
In that case we expect the limit $\eps\to0$ to lead to a two-phase  free boundary problem (Mullins-Sekerka).\medskip

While it is not difficult to extend Theorem  \ref{thm:HS} to general nonlinearities, the proofs provided in \cite{KMW2,M23} are strongly dependent on the particular form of the interaction kernel $G$. In fact the proofs make use of the elliptic equation \eqref{eq:G}. 
One could extend the proof to slightly more general equations, for example by replacing the Laplace operator by a general elliptic operator (which would introduce both space and direction dependence in the free boundary condition in \eqref{eq:HS}), but treating the case of general interaction kernels seems to require new ideas.
As noted earlier, the $\Gamma$-convergence of the energy (Theorem \ref{thm:Gamma}) is known to hold for a large class of such kernel, and the main obstacle is rather to prove Proposition \ref{prop:firstvar} in  a more general framework.  
\medskip

\section{Bounded domains and contact angle conditions}\label{sec:ca}
We have so far only considered problems set in the whole space.
But in many experimental settings, the particles are confined to a bounded domain $\Omega \subset \R^d$. The equation for $\rho$ is then set on $\Omega$ and supplemented with no-flux boundary conditions on $\pa\Omega$. 
But the domain can also have a non-trivial influence on the interaction kernel.

In some settings, it might make sense to stick to interactions that depend only on the distance between the particles (mathematically, this is done by extending $\rho$ by zero outside the domain and keeping the convolution $G_\eps * \rho$ in the equation).
But in the context of chemotaxis \eqref{eq:KellerSegel}, the potential $\phi=\phi_\eps =G_\eps*\rho$ represents the concentration of a chemoattractant chemical which is now diffusing in the same bounded domain $\Omega$ so that this convolution should be replaced with the solution of some elliptic boundary value problem.

In the first part of this section, we present results that have been proved in this later setting. We will then discuss   the case of interactions that depend only on the distance between the particles.
\medskip

\subsection{Chemotaxis potential in bounded domains}
In the context of chemotaxis \eqref{eq:KellerSegel},
we replace the convolution $G_\eps*\rho$ by the solution $\phi_\eps$ of the following elliptic boundary value problem:
\begin{equation}\label{eq:phieps}
\begin{cases}
\sigma \phi_\eps-\eps^2 \Delta \phi_\eps = \rho, & \mbox{ in } \Omega;  \\
a \phi_\eps +\eps b \na \phi_\eps\cdot n=0, & \mbox{ on }\pa \Omega.
\end{cases}
\end{equation}
We are taking general Robin type boundary conditions at the boundary $\pa\Omega$ for further generality, even though Neumann conditions ($a=0$) are probably the most natural type of conditions. When $a>0$ \eqref{eq:phieps} takes into account the possible leakage (or destruction) of the chemical at the boundary.
We can take $a$ and $b$ to be constant or functions of $x\in\pa\Omega$ and 
we will always assume that $a\geq 0$, $b\geq 0$ and $a+b>0$.
The solution of \eqref{eq:phieps} can be written as  $\phi_\eps(x) = \int_\Omega G_\eps(x,y)\rho(y)\, dy$ with $G_\eps(x,y)\geq 0$ a symmetric kernel
and we set
$$
B_\eps[\rho] :=  \phi_\eps  - \sigma^{-1}\rho $$
so that \eqref{eq:NCH2} corresponds to
\begin{equation}\label{eq:NCH2bd}
\begin{cases}
\eps \pa_t \rho  = \div(\rho\na h'(\rho) )-  \div(\rho \na B_\eps[\rho]), & \mbox{ in } (0,T)\times \Omega;\\
\rho \na (h'(\rho)- B_\eps[\rho])\cdot n = 0, & \mbox{ on }  (0,T)\times\pa \Omega;\\
\rho(x,0) = \rho_{in}, & \mbox{ in } \Omega.
\end{cases}
\end{equation}

\medskip

However, we do not have $\int_{\Omega} G_\eps(x,y)\, dy =\sigma^{-1}$ so we cannot write $B_\eps[\rho]$ as $\int_\Omega G_\eps(x,y)[\rho(y)-\rho(x)]\, dy$. 
We can still write 
$ \int_{\Omega} G_\eps(x,y)\, dy = \sigma^{-1} - \tau_\eps(x)$
where $\tau_\eps\geq 0$ solves
$$
\begin{cases}
\sigma \tau_\eps - \eps^2 \Delta\tau_\eps =0 , & \mbox{ in } \Omega;  \\
a \tau_\eps +\eps b \na \tau_\eps\cdot n= \frac a \sigma, & \mbox{ on }\pa \Omega.
\end{cases}
$$
We thus have 
$$ B_\eps[\rho] (x) = \int_\Omega G_\eps(x,y)[\rho(y)-\rho(x)]\, dy - \tau_\eps(x) \rho(x)$$
and the corresponding energy functional can be written as
\begin{align*}
\G_\eps[\rho]  & =   \ds\frac{1}{\eps} \int_{\Omega} h(\rho) \, dx - \frac{1}{2\eps} \int_{ \Omega }   \rho B_\eps[\rho] \, dx\nonumber \\
& =   \ds\frac{1}{\eps} \int_{\Omega} h(\rho) \, dx +\frac{1}{4\eps} \iint_{\Omega\times \Omega }  G_\eps(x,y)[\rho(x)-\rho(y)]^2\, dx\,dy+ \frac{1}{2\eps} \int_{ \Omega }  \tau_\eps(x)  |\rho(x)|^2 \, dx.
\end{align*}
The last term in this energy  accounts for the boundary conditions on $\phi_\eps $. 
In the case of Neumann boundary condition $a=0$, we have $\tau_\eps=0$, and this new term vanishes. But in general this new term will contribute to the limiting energy.
Indeed, when $\pa\Omega$ is regular enough, we can show\footnote{To see this, notice that when $\Omega$ is the half space $\{x_d>0\}$, we have $\tau_\eps(x) =\frac{1}{\sigma} \frac{a}{a+b\sqrt\sigma} e^{-\frac{\sqrt \sigma}{\eps}x_d}$ and so $\eps^{-1}\tau_\eps \to \frac{1}{\sigma^{3/2}} \frac{a}{a+b\sqrt\sigma}  \delta(x_d)$.} that $\frac 1 \eps\tau_\eps$  converges as a distribution to  $\frac{1}{\sigma^{3/2}} \frac{a}{a+b\sqrt\sigma} d\mathcal H^{n-1}|_{\pa\Omega}$.
This suggests that if $E$ is a subset of $\Omega$ with $C^1$ boundary, then
$$\G_\eps[\theta \chi_E] \to \gamma\theta P(E,\Omega)   + \frac{1}{2 \sigma^{3/2}} \frac{a \theta }{a+b\sqrt\sigma} \mathcal H^{n-1}(\pa\Omega \cap E),$$
where the local perimeter $P(E,\Omega)$ is defined by
$$
P(E,\Omega) := \sup\left\{\int_E \div g\, dx\, ;\, g\in (C^{1}_0( \Omega))^d, \quad |g(x)|\leq 1 \; \forall x\in \Omega\right\}  = \int_{\Omega} |\na\chi_E|.
$$
This limit was proved rigorously in \cite{MW} (see also \cite[Appendix C]{KMW2}) for the hard-sphere case $f=f_\infty$. In that case $\theta=1$ and $\gamma = \frac{1}{4 \sigma^{3/2}}$, so the formula for the limit becomes
\begin{equation}\label{eq:nlsc}   \frac{1}{4 \sigma^{3/2} }\left( P(E,\Omega)  +   \frac{2a}{a+b\sqrt\sigma} \mathcal H^{n-1}(\pa\Omega \cap E)\right).
\end{equation}
Here we make a couple of remarks: First, functionals like \eqref{eq:nlsc}, in which the area of the free surface $\pa E\cap \Omega$ and the contact area $\pa\Omega\cap E$ are weighted differently, appear naturally in various applications, such as the capillary energy for the modeling of sessile liquid drop   in contact with a solid support. It is well known that the minimization of such energies leads to contact angle conditions. Second, we note that when $ \frac{2a}{a+b\sqrt\sigma} >1$, the energy \eqref{eq:nlsc} is not lower semi-continuous and can therefore not be the $\Gamma$-limit of $\G_\eps$.
In fact, the following result was proved in \cite{MW,KMW2} when $f=f_\infty$:
\begin{theorem}\label{thm:4}
Assume that $\Omega$  is a bounded open set with $C^{1,\alpha}$ boundary and take $f=f_\infty$ and $G$ given by \eqref{eq:G}.  
Assume further that $a$ and $b$ are non-negative constants such that $a+b>0$.
The functional $ \G_\eps $ $\Gamma$-converges, when $\eps\to 0$ to 
$$ \G_{0 }(\rho) 
:=
\begin{cases}
\displaystyle  \frac {1}{4 \sigma^{3/2} } \left[  \int_\Omega |\na \rho|  + \int_{\pa \Omega } \min\left( 1  ,\frac{2 a}{a+\sqrt \sigma b} \right) \rho \, d\H^{n-1}(x) \right] & \mbox{ if } \rho \in \BV(\Omega;\{0,1\});\\
\infty & \mbox{ otherwise.}
\end{cases}
$$
\end{theorem}
The result also holds when $a(x)$, $b(x)$ are bounded, Lipschitz non-negative functions defined on $\pa\Omega$ such that $\alpha(x)+\beta(x) \geq \eta>0$ and  that 
$\H^{n-2}\left(\partial \left\{\frac{2a}{a+\sqrt\sigma b}>1 \right\}\right)<\infty$ (this technical condition ensures that the function $x\mapsto \frac{2a(x)}{a(x)+\sqrt\sigma b(x)}$ does not oscillate too much around the value $1$). 
 A crucial tool for the proof of this result is the following alternative formula for $\G_\eps$:
\begin{align*}
 \G_\eps (\rho)  & =  \frac{1}{2  \eps}\int_\Omega h_\infty(\rho) \, dx
+  \frac{1}{2\sigma \eps}  \int _{\Omega}(\rho-\sigma \phi_\eps )^2 \, dx  + 
\frac{ \eps}{2}  \int_\Omega |\na \phi_\eps | ^2\, dx\nonumber \\
 & \qquad + \frac{\eps^2}{2}  \int_{\pa \Omega}\frac{b}{a+b} |\na \phi_\eps \cdot n|^2 \, d\H^{n-1}(x) +
\frac{1}{2}  \int_{\pa \Omega}\frac{a}{a+b} |\phi_\eps |^2 \, d\H^{n-1}(x), 
 \end{align*}
where $\phi_\eps (x)=\int_\Omega G_\eps(x,y) \rho(y)\, dy$ solves \eqref{eq:phieps}.

Theorem \ref{thm:4} was also proved in  \cite{M23} for the soft-sphere case $f=f_m$ with $m>2$  when $a=0$ (Neumann boundary condition).  The approach of  \cite{MW} and  \cite{M23} can be combined to get the result when $f=f_m$ and for general coefficients $a$ and $b$ (although the formula for the coefficient in the boundary integral is not straightforward).

When $a=0$ then $\G_0(\chi_{E})$ is proportional to the local perimeter $P(E ,\Omega)$ and thus only measures the area of $\pa E\cap \Omega$, while when $b=0$ (Dirichlet boundary conditions), then $\G_0$ is proportional to the full perimeter $P(E)$. 
When  $\frac{2 a}{a+\sqrt \sigma b} \in(0,1)$, the areas of $\pa E\cap\Omega$ and $E\cap \pa\Omega$ are weighted differently.

\medskip

As mentioned above, the boundary term  in $\G_0$ translates into a contact angle condition in the limiting Hele-Shaw free boundary problem:
It was proved in \cite{KMW} that the limit $\eps\to0$ of \eqref{eq:phieps}-\eqref{eq:NCH2bd} leads to the same Hele-Shaw free boundary problem \eqref{eq:HS} in $\Omega$,  supplemented with 
Neumann boundary condition
$$
\na p\cdot n = 0 \qquad\mbox{ on } \pa\Omega$$
and 
the condition (see Figure \ref{fig:a})
 \begin{equation}\label{eq:CAC}
\cos(\alpha):=-\min\left( 1,\frac{2a}{a+\sqrt \sigma b} \right) \,\,\hbox{ on } \pa E(t) \cap \pa\Omega,
\end{equation}
where $\alpha$ denotes the angle between the free surface $\pa E$ and the fixed boundary $\pa\Omega$ (see Figure \ref{fig:a}).
\begin{figure}[H]
 	 		\includegraphics[width=.4\textwidth]{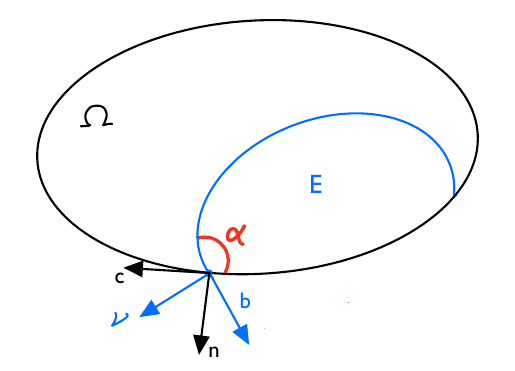}			
\vspace*{-20pt}
 	 		\caption{The contact angle condition}
 	 		\label{fig:a}
 	 	\end{figure}
In particular, for Neuman condition $a=0$ (no absorption of the chemical on $\pa\Omega$) we get $\alpha = \frac{\pi}{2}$ so the contact must be orthogonal, while for Dirichlet condition $b=0$ (and whenever the absorption rate satisfies $a/b \geq \sqrt \sigma $), the contact must be tangential ($\alpha=\pi$).
Let us note that a similar contact angle condition was derived recently in \cite{Laux2} for the geometric mean-curvature flow.

The precise result  in this case, which is proved by establishing the equivalent of Proposition \ref{prop:firstvar} with the boundary term (see \cite[Proposition 5.2]{KMW2}), is as follows:
\begin{theorem}\cite[Theorem 1.6]{KMW2}
\label{thm:HSbd}
Under the assumptions of Theorem \ref{thm:4}, and given a sequence $\eps_n$ and initial data $\rho_{in}^{\eps_n}$ such that $\G_{\eps_n}(\rho_{in}^{\eps_n})\leq M$ and $\rho_{in}^{\eps_n}\to \rho_{in} =  \chi_{E_{in}} \in BV(\Omega)$ as $\varepsilon_n \to 0$,  let $\rho^{\eps_n}(t,x)$ be the weak solution 
of \eqref{eq:NCH2bd}. 

Then the conclusion of Theorem \ref{thm:HS} holds, with \eqref{eq:weakp} replaced by
\begin{align}\label{eq:weak12b}
\int_0^\infty \int_\Omega  \rho v \cdot \xi  - \theta p  \, \div \xi(t) \,dx\,  dt
& = -\gamma \int_0^\infty \int_\Omega \left[  \div \xi - \nu\otimes \nu :D\xi\right] |\na \rho|\, dt \\
& \qquad  -\gamma \min\left( 1,\frac{2a}{a+\sqrt \sigma b} \right)  \int_0^\infty \int_{\pa\Omega} \left[  \div \xi - n\otimes n :D\xi\right]  \rho\,  d\mathcal H^{n-1}(x)\, dt,\nonumber 
\end{align}
for any vector field  $\xi \in C_c^\infty( (0,\infty)\times \overline \Omega ; \R^d)$ such that $\xi \cdot n=0$ on $\pa\Omega$. 
\end{theorem}
The fact that  \eqref{eq:weak12b} implies the  contact angle condition  \eqref{eq:CAC}  can be seen by using the classical formula (for a smooth interface $\Sigma$):
$$
\int_\Sigma \div \xi - \nu\otimes\nu : D\xi   = \int_\Sigma \kappa\, \xi\cdot\nu + \int_\Gamma b\cdot \xi,
$$
where $\nu$ is the normal vector to $\Sigma$,  $\kappa$ denotes the mean curvature of $\Sigma$ and $b$ is the conormal vector along $\Gamma = \pa \Sigma$.
Indeed, formally at least, the right hand side of   \eqref{eq:weak12b} is (using the fact that $\xi\cdot n=0$ on $\pa\Omega$):
\begin{align*}
& - \gamma  \left[  \int_{\Sigma} \kappa \xi\cdot\nu + \int_\Gamma \vec b \cdot \xi \right] + \gamma\min\left( 1,\frac{2a}{a+\sqrt \sigma b} \right)  \left[  \int_{\pa\Omega \cap \Omega_s} \kappa \xi\cdot n + \int_\Gamma \vec c \cdot \xi \right]\\
& =
 \gamma \left[  - \int_{\Sigma} \kappa \xi\cdot\nu +\int_\Gamma  \min\left( 1,\frac{2a}{a+\sqrt \sigma b} \right) \vec c \cdot \xi -\vec b \cdot \xi\right],
\end{align*}
where $\vec b$ and $\vec c$ are unit conormal vectors along $\Gamma = \pa \Sigma \cap \pa\Omega$: $\vec b$ is tangent to $\Sigma$ while $\vec c$ is tangent to $\pa\Omega$ (see Figure \ref{fig:a}).
When $\xi = \vec c$ on $\pa\Omega$, we deduce
$\cos \alpha = \vec b \cdot \vec c = \min\left( 1,\frac{2a}{a+\sqrt \sigma b} \right) $. 

\medskip

\subsection{Interactions dependent only on the distance in a bounded domain}\label{sec:CAB}
While the framework discussed above makes sense in the context of chemical interactions, there are other settings in which interactions depend on the distance between particles, even in the presence of fixed boundaries (or obstacles). This is often the case of social interactions which are facilitated by sight \cite{BT16}. 
Another example is the modeling of cell-cell adhesion \cite{CMSTT,FBC} by an attractive force between cells that are within a given distance from each others and the formation of membraneless organelles in eukaryotic cells  \cite{HN,MK} (in that case, the bounded domain is the cell).
In that case, we have
$$ 
\phi_\eps (x) = \int_\Omega G_\eps(x-y) \rho (y)\, dy =  \int_{\R^d} G_\eps(x-y) \bar \rho (y)\, dy 
$$
where $\bar\rho$ is the extension of $\rho$ to $\R^d$ by $0$.
For simplicity, we focus on the hard-sphere case $f=f_\infty$, although the soft-sphere case could be treated similarly.
The energy functional then reads 
$$
\eps^{-1 } \int_\Omega h_\infty(\rho(x))   - \frac 1 2  \rho(x) \phi_\eps  (x)\, dx =\eps^{-1} \int_{\R^d} h_\infty(\bar \rho(x)) - \frac 1 2 \bar \rho(x) G_\eps * \bar \rho(x)\, dx 
$$
which is the same energy as studied in Section \ref{sec:st} applied to $\bar \rho$ instead of $\rho$.
Its $\Gamma$-limit is proportional to the (full) perimeter functional $ P(E)$. The limit $\eps\to0$ of the aggregation-diffusion equation \eqref{eq:LSSbd} then leads to the Hele-Shaw free boundary problem \eqref{eq:HS} supplemented with the contact angle condition
$$ \cos\alpha = -1$$
that is $\alpha=\pi$: The boundary of the domain  $\pa\Omega$ acts like an hydrophobic surface and any contact of the free surface with the fixed boundary  will be tangential.

\medskip

Other contact angle conditions may be derived by taking into account the interactions of the particles with the boundary of the domain (in the modeling of cell-cell adhesion, the fixed boundary may represent a different species with a different inter-species adhesion strength).
To illustrate this effect,  we introduce the potential $\tau_\eps(x) = G_\eps * \chi_{\R^d\setminus\Omega}$
and consider the equation
\begin{equation}\label{eq:LHSeo}
\begin{cases}
\eps \pa_t \rho + \div(\rho \na (G_\eps * \rho ) ) + \eta  \, \div(\rho \na \tau_\eps) = \div( \rho  \na   p), \qquad p\in \pa f_\infty(\rho), & \mbox{ in } [0,\infty)\times\Omega;  \\
\rho  \na [G_\eps * \rho   + \eta  \tau_\eps- p]\cdot n = 0, & \mbox{ on }[0,\infty)\times \pa \Omega;  \\
\rho(x,0) = \rho^\eps_{in}(x), & \mbox{ in } \Omega.
\end{cases}
\end{equation}
The new drift term appearing in this equation models  the attraction of the particles toward $\R^d\setminus \Omega$ (or repulsion if $\eta<0$).
It describes the interactions of our particles with (different) particles located outside of $\Omega$. These interactions are assumed to be  of similar nature but possibly different strength (another classical example is that of a fluid resting on a solid surface: The fluid molecules interact with each others as well as with the solid). 

The respective strength of the interactions - the coefficient $\eta$ - determines the contact angle condition. 
This can be seen by considering the energy associated to \eqref{eq:LHSeo}, which reads:
$$
\G_{\eta,\eps}(\rho)= \eps^{-1} \int_\Omega  f_\infty(\rho)   + \frac{1}{2\sigma}\rho -\frac 1 2 \rho\phi_\eps -\eta\rho\tau_\eps
\, dx.
$$
Using the fact that $\tau_\eps(x) = \frac1  \sigma - \int_\Omega G_\eps(x-y)\, dy$, we can write 
\begin{align*}
\G_{\eta,\eps}(\rho)
 & = \eps^{-1} \int_\Omega   f_\infty(\rho) +\frac{1}{2} \phi_\eps  (1-\rho)  -\eta\rho\tau_\eps +\frac 1 2\left(\frac 1 \sigma \rho- \phi_\eps  \right)\, dx\\
 & = \eps^{-1} \int_\Omega  f_\infty(\rho) +\frac{1}{2} \phi_\eps  (1-\rho) \, dx + \left(\frac 1 2 -\eta\right) \eps^{-1} \int_\Omega \rho \tau_\eps\, dx.
\end{align*}
This energy has two terms: The first term was studied for example in \cite{JKM} where it was shown (for a different kernel, but the proof extends to \eqref{eq:G} using the arguments of \cite{KMW2}) that 
it $\Gamma$-converges to the local perimeter 
$$ \frac{1}{4\sigma^{3/2}} P(E,\Omega).$$ 
The second term yield a contribution from the boundary.
 When $G$ is the Newtonian attractive kernel \eqref{eq:G}, then
$\tau_\eps$ solves
\begin{equation}\label{eq:tau1}
\tau_\eps-\eps^2\Delta \tau_\eps = \chi_{\R^d\setminus \Omega}.
\end{equation}
When $\Omega $ is the half space $\{x \in \R^d\,; \, x_d>0\}$,  we can solve \eqref{eq:tau1} explicitly and find $\tau_\eps (x)= \frac{1}{2\sigma} e^{-\frac{\sqrt\sigma x_d}{\eps}}$ for $x_d\geq 0$ so that $\frac 1 \eps \tau_\eps(x) \chi_{x_d>0}$ converges to $\frac{1}{2\sigma^{3/2}} \delta(x_d)$ as $\eps\to0$ and 
$$\eps^{-1} \int_{\Omega} \rho \tau_\eps\, dx \to \frac{1}{2\sigma^{3/2}} \int_{\pa\Omega} \rho\, d\mathcal H^{n-1}(x).$$
This convergence can be justified for general domains $\Omega$ with smooth boundary. In fact, we can show:
\begin{proposition}\label{prop:Geta}
Assume that $\Omega$  is a bounded open set with $C^{1,\alpha}$ boundary and take $f=f_\infty$ and $G$ given by \eqref{eq:G}. For all $\eta\in[0,1/2]$,  the energy functional $\G_{\eta,\eps}$ $\Gamma$-converges to 
 \begin{equation}\label{eq:energygta}
 \frac 1 {4\sigma^{3/2}} P(E,\Omega) + \frac 1 {2\sigma^{3/2}} \left(\frac 1 2 -\eta\right) H^{n-1}(\pa\Omega \cap E).
 \end{equation}
\end{proposition}
Note that the result holds for $\eta\in[0,1)$ but the proof is more delicate than the proof given below when $\eta>1/2$.
 This Proposition suggests that a result similar to Theorem \ref{thm:HSbd} holds for \eqref{eq:LHSeo}:
 Under some assumptions on the convergence of the energy, the limit $\eps\to 0$ in \eqref{eq:LHSeo} leads to the 
 Hele-Shaw free boundary problem (with surface tension) \eqref{eq:HS}, supplemented with the contact angle condition
$$ \cos \alpha = 2\eta -1$$
at the triple junction. 
We can thus have
contact angles between $\pi $ and $0$ when $\eta $ takes values between $0$ and $1$. When $\eta\geq 1$, the attraction towards the boundary dominates the dynamics and the  particles will spread over the whole boundary to maximize the contact area.
A rigorous justification of this result would require proving the convergence of the first variation of the energy (the equivalent of Proposition \ref{prop:firstvar}).

\begin{proof}[Idea of the proof of Proposition \ref{prop:Geta}]
We note that 
$$ 
\int_\Omega \rho\tau_\eps \, dx
=\int_\Omega \rho*G_\eps \chi_{\R^d\setminus \Omega}  \, dx=
\int_\Omega \phi_\eps  \chi_{\R^d\setminus \Omega}  \, dx
=\int_\Omega \frac 1 \sigma \rho - \phi_\eps    \, dx
$$
and we write
\begin{align*}
\G_{\eta,\eps}(\rho)
& =
 \eps^{-1}(1-2\eta)  \int_\Omega  f_\infty(\rho)   + \frac{1}{2\sigma}\rho -\frac 1 2 \rho\phi_\eps  
\, dx  +  \eps^{-1} 2\eta \int_\Omega  f_\infty(\rho)   + \frac{1}{2\sigma}\rho -\frac 1 2 \rho\phi_\eps -\frac 1 2 \rho\tau_\eps
\, dx\\
& =
 \eps^{-1}(1-2\eta)  \int_{\R^d}  h_\infty(\bar \rho)   + \frac 1 2 \bar \rho ( 1-\phi_\eps )  
\, dx  +  \eps^{-1} 2\eta \int_\Omega  f_\infty(\rho)    + \frac 1 2 \phi_\eps  (1-\rho) 
\, dx
\end{align*}
where every term is non-negative.
The first term is (up to a multiplicative constant) the energy $\G_\eps$ and thus $\Gamma$-converges to  $\frac 1 {4\sigma^{3/2}} (1-2\eta)  P(E)$ (when $\eta\leq 1/2$).
The second term can be written as
$$
\eps^{-1} 2\eta \int_\Omega  f_\infty(\rho)  +  \frac 12 \tau_\eps(x)\rho(x)(1-\rho(x))  \, dx + \frac 1 4 \iint_{\Omega\times\Omega} G(x-y) [\rho(y)-\rho(x)]^2\, dx\,dy.
$$
where the first integral is the double well-potential and the second integral is similar to the energy studied, for example, in \cite{AB98}. This energy $\Gamma$-converges to
$2\eta  \frac 1 {4\sigma^{3/2}} P(E,\Omega)$.
The limiting energy is thus 
$$\frac 1 {4\sigma^{3/2}} (1-2\eta)  P(E) + 2\eta  \frac 1 {4\sigma^{3/2}} P(E,\Omega)$$ which is equivalent to \eqref{eq:energygta}.
\end{proof}

\section{Conclusion}

In this paper, we  presented a road map to investigate phase separation phenomena in biological (and other) systems. We start from an agent-based microscopic description of attractive behavior limited by a volume exclusion principle, accounting for the natural incompressibility of the agents.
The key  takeaway is that phase separation is natural when repulsion is modeled by nonlinear diffusion with $m>2$ (soft-sphere models) or with density constraint (hard-sphere models), and when the long-range attractive kernel is integrable.

Based on recent works, we proposed a rigorous mathematical approach to derive effective models of geometric nature which describe the evolution of the sharp interface over appropriate time scales after aggregation occurs. In particular, the results presented here show that surface tension and contact angle dynamics, classically associated with fluid interfaces, are natural consequences of volume-limited aggregation.
This is in agreement with many real-life observations: for instance, the  liquid-like behavior of
membraneless organelles in eukaryotic cells has been well documented \cite{B09,HN}.

\medskip

This review  leads to several important open questions. Among them is the 
well-posedness of the hard-sphere microscopic model \eqref{eq:NLHS} for general initial conditions and its convergence as $\delta\to 0$ to the macroscopic model \eqref{eq:LHS}.
Solutions to \eqref{eq:incompdelta} can be formally obtained in the limit $m\to\infty$ in the soft-sphere model \eqref{eq:NLSS} or via a JKO type discrete scheme, but neither approaches has been made rigorous. Some of the issues encountered in the latter one are discussed in further details in  \cite{KMJW}.

Concerning the derivation of the free boundary problems, we emphasize once more that all stated rigorous results are conditional, which require an  assumption on the energy convergence. Obtaining non-conditional convergence results remains  out of reach, not only in the context of the problems discussed in this review, but for related free boundary problems as well.
Last but not least, these conditional results have been proved only for the particular kernel $G$, solution of \eqref{eq:G}. On the other hand,
 the $\Gamma$-convergence of the energies $\J_\eps$ and $\G_\eps$ has been proved for very general $G$. Thus we may  conjecture that these conditional results can hold for a large class of non-negative integrable kernels.
 Further potential generalizations of our analysis also include the derivation of  two-phase free boundary problems, either by considering pressure function $f$ for which the two wells of the function $h$ are both positive (leading to a Mullins-Sekerka free boundary problem), or for multi-species models such as those proposed in \cite{FBC}.

\appendix

\section{From micro to macro: convergence $\delta\to 0$} \label{app:conv}
\begin{proof}[Proof of Theorem \ref{thm:convdelta}]
The continuity equation and \eqref{eq:energydeltaine} imply that 
$\rho^\delta$ is bounded  $C^{1/2}([0,T];\mathcal P_2(\R^d))$. It follows from a refined version of Ascoli-Arzel\`a theorem, as in \cite[Proposition 4.1]{carrillo2023nonlocal}, 
that there exists a subsequence such that   $\rho^\delta(t)$ narrowly converge to some function $\bar\rho(t)$ (uniformly in $t$).

Since $K_\delta$ is supported in $B_\delta$, we have $K_\delta * \vphi \to \vphi$ uniformly for any uniformly continuous function $\vphi$. Using the bound on the second moment \eqref{eq:second}, we can then show that $K_\delta * \rho^\delta (t) $ converges narrowly to $\bar \rho(t)$.

We also note that \eqref{eq:energydeltaine} implies that $\rho^\delta*K_\delta$ is bounded in $L^\infty(0,T;L^m(\R^d))$ and thus also in $L^\infty(0,T;L^2(\R^d))$.
Furthermore, using the fact that $\phi^\delta  = G* \rho^\delta$ solves $\sigma \phi^\delta - \eta \Delta \phi^\delta = \rho^\delta$,
we see that $\phi^\delta$ and $K_\delta *\phi^\delta$ are bounded in $ L^\infty(0,T;H^1(\R^d))$.

However, this convergence is not enough to pass to the limit in the equation and the main idea in  \cite{carrillo2023nonlocal} is to rewrite the limiting equation \eqref{eq:LSS} (when $f$ is given by \eqref{eq:fm}) as
\begin{equation}
	\label{eq:pme_wk_str}
\partial_t \rho = - \div(\rho\na G*\rho) + 2\nabla\cdot (\rho^\frac{m}{2}\na \rho^\frac{m}{2})
\end{equation}
and to adjust, in a similar way, the weak formulation of the regularized equation \eqref{eq:NLSS}. Indeed, given a test function $\psi \in C_c^\infty([0,T]\times \R^d)$,  a straightforward computation with integration by parts yields
\begin{align}
\int & \psi(T,x)\, d\rho^\delta(T,x)  - \int \psi(0,x)\, d\rho_{\text{in}}(x) - \int_0^T\int\partial_t\psi \,d\rho^\delta(t,x)\, dt 	\notag\\
&= 
\int_0^T\int    \na (K_\delta*K_\delta*\phi^\delta) \cdot \na \psi \, d\rho^\delta(t,x)\, dt 
-\frac{m}{m-1}\int_0^T\int \na K_\delta*(K_\delta*\rho^\delta)^{m-1}\cdot \na \psi \, d\rho^\delta(t,x)\, dt 	\notag\\
&= \int_0^T\int \na (K_\delta*\phi^\delta) \cdot [K_\delta*(\rho^\delta \, \na \psi)]   
-m\int_0^T\int (K_\delta*\rho^\delta)^{m-2}\nabla (K_\delta*\rho^\delta)\cdot [K_\delta*(\rho^\delta \, \na \psi)] 	\notag\\
&= \int_0^T\int (K_\delta*\rho^\delta) \na (K_\delta * \phi^\delta)  \cdot \na \psi 	  -2\int_0^T\int (K_\delta*\rho^\delta)^\frac{m}{2}\na (K_\delta*\rho^\delta)^\frac{m}{2}\cdot \na \psi 	\label{eq:pme_wk_str_delta}\\
&\quad 
+ \int_0^T\int  \na (K_\delta * \phi^\delta) \cdot E^\delta
-2\int_0^T\int (K_\delta*\rho^\delta)^{\frac m 2 -1}\na(K_\delta*\rho^\delta)^{\frac{m}{2}}\cdot E^\delta. \label{eq:pme_delta_error}
\end{align}
The error term in~\eqref{eq:pme_delta_error} comes from the commutation between convolution and multiplication, i.e. $E^\delta(x) := [K_\delta*(\rho^\delta \na \psi)](x) - (K_\delta * \rho^\delta)\na \psi(x)$.
We will see that this error term vanishes in the limit. Passing to the limit in the second term of~\eqref{eq:pme_wk_str_delta} (which has the same structure as \eqref{eq:pme_wk_str}) requires the following convergences:
\[
(K_\delta*\rho^\delta)^\frac{m}{2}\to \rho^\frac{m}{2}, \quad\text{(strong) in }L_{t,x}^2,\quad \text{and}\quad \na(K_\delta*\rho^\delta)^\frac{m}{2}\rightharpoonup \na \rho^\frac{m}{2}, \quad \text{(weak) in }L_{t,x}^2.
\]
This weak-strong convergence pair is proved by considering the evolution of $\int \rho^\delta \log \rho^\delta$. Indeed, since $\E [\rho_{\text{in}}]<+\infty$ and $\rho\log \rho \lesssim \rho^m$ for $m>1$, we see that $\int \rho_{\text{in}}\log \rho_{\text{in}}<+\infty$. 
Formally (these steps can be made rigorous by using the piecewise constant interpolation given by the JKO scheme as in \cite{carrillo2023nonlocal}), by differentiating $\int \rho^\delta \log \rho^\delta$ in time  we get:
\[
\int \rho^\delta(t,x)\log\rho^\delta(t,x) - \int \rho_{\text{in}}\log \rho_{\text{in}} = - \frac{2}{m}\int_0^t\int |\na (K_\delta *\rho^\delta)^\frac{m}{2}|^2 + \int_0^t\int \na (K_\delta*\phi^\delta) \cdot\na K_\delta*\rho^\delta
\]
where
\begin{align*}
 \int_0^t\!\! \!\int \na (K_\delta*\phi^\delta) \cdot\na K_\delta*\rho^\delta
& = - \int_0^t\!\! \! \int \Delta (K_\delta*\phi^\delta) \, K_\delta*\rho^\delta =- \frac{\sigma}{\eta}\int_0^t \!\! \! \int  (K_\delta *\phi^\delta)  \, K_\delta*\rho^\delta + \frac{1}{\eta} \int_0^t\!\! \! \int  | K_\delta*\rho^\delta|^2 
\end{align*}
and so 
$$
\int \rho^\delta(t,x)\log\rho^\delta(t,x) +
\frac{2}{m}\int_0^t\!\! \! \int |\na (K_\delta *\rho^\delta)^\frac{m}{2}|^2 + 
\frac{\sigma}{\eta}\int_0^t\!\! \! \int  (K_\delta*\phi^\delta) \, K_\delta*\rho^\delta
= \int \rho_{\text{in}}\log \rho_{\text{in}} +\frac{1}{\eta}\int_0^t\!\! \! \int  | K_\delta*\rho^\delta|^2.
$$
We recall (see \cite[Remark 4.2]{carrillo2023nonlocal}) that the entropy  is bounded below by the second moment as follows:
$$\int \rho^\delta(t,x)\log\rho^\delta(t,x) \geq -d\log(2\pi)-\frac 1 2 \int |x|^2\rho^\delta$$
and since $m>2$, we have
\begin{align*}
\int_0^t\int  | K_\delta*\rho^\delta|^2
& \leq C  \int_0^t\int K_\delta*\rho^\delta + |K_\delta*\rho^\delta|^m\, dx\\
& \leq C \left( 1 +  \E_\delta[\rho_{\text{in}}] \right)t.
\end{align*}
This implies that the sequence $(K_\delta *\rho^\delta)^\frac{m}{2}$ is bounded in $L^2(0,T;H^1(\R^d))$ uniformly in $\delta>0$.
Proceeding as in \cite[Proposition 4.3]{carrillo2023nonlocal} (using a refined version of the Aubin-Lions Lemma \cite[Theorem 2]{RS03}) we can then show that up to another subsequence we have
$$ 
K_\delta *\rho^\delta \to \rho \qquad\mbox{ in } L^m([0,T]\times\R^d), \qquad \na (K_\delta *\rho^\delta)^{\frac m 2} \rightharpoonup\na \rho^{\frac m 2}  \mbox{ in } L^2([0,T]\times\R^d).
$$
We already mentioned that $K_\delta*\phi^\delta$ was bounded in $ L^\infty(0,T;H^1(\R^d))$, which implies
$$ \na (K_\delta*\phi^\delta) \rightharpoonup \na \phi \mbox{ in } L^2([0,T]\times\R^d).$$
Finally, (see \cite[Lemma 5.1]{carrillo2023nonlocal}) we have that $E^\delta(t,x)$ converges to zero in $L^m([0,T]\times\R^d)$.
These convergences are enough to show that~\eqref{eq:pme_delta_error} vanishes as $\delta\to 0$ and to pass to the limit in  
~\eqref{eq:pme_wk_str_delta}.
\end{proof}

\section{Convergence results for the Stefan problem}\label{app:Gamma}
\begin{proof}[Proof of Proposition \ref{prop:Gamma1}]
We have $ \J_\eps[\rho^\eps] \geq \J^*[\rho^\eps]$ and the functional $\J^*$ is convex and thus lower semicontinuous with respect to the weak topology. The liminf property follows.

To prove the limsup property, we will simply use the fact that $\J_\eps$ can be bounded by some Allen-Cahn functional for which the $\Gamma$ convergence has been proved elsewhere. Indeed,
we can write
$$
\int\!\!\!\int_{\R^d\times\R^d} G_\eps(x-y)[\rho(x)-\rho(y)]^2\, dx\,dy = \int\!\!\!\int_{\R^d\times\R^d} G_\eps(z)[\rho(y+z)-\rho(y)]^2\, dz\,dy
$$
where
\begin{align*}
 |\rho(y+z) -\rho(y) |^2 
 & = \left| \int_0^1 \na \rho(y+tz) \cdot z \, dt \right|^2\\
 & \leq   \int_0^1 \left| \na \rho(y+tz)\right|^2   \, dt\,  |z|^2  .
 \end{align*}
 We deduce
 \begin{align*}
 \int\!\!\!\int_{\R^d\times\R^d} G_\eps(x-y)[\rho(x)-\rho(y)]^2\, dx\,dy
& \leq  \int_0^1  \int\!\!\!\int_{\R^d\times\R^d} |z|^2 G_\eps(z)\left| \na \rho(y+tz)  \right|^2  \, dz\,dy \, dt\\
& \leq  \int_0^1   \int_{\R^d} |z|^2 G_\eps(z) \int_{\R^d}  \left| \na \rho(y)  \right|^2  \, dy \, dz  \, dt\\
& \leq  \eps^2   \int_{\R^d} |z|^2 G(z) \,dz   \int _{\R^d} \left| \na \rho(y)  \right|^2  \, dy .
\end{align*}
We thus have  
$$ \J_\eps[\rho^\eps] \leq  \int_{\R^d} h(\rho^\eps) \, dx + \eps^2 D  \int _{\R^d} \left| \na \rho^\eps(x)  \right|^2  \, dx$$
for some constant $D$. The functional in the right hand side is an Allen-Cahn functional for which the $\Gamma$ convergence is proved, for example in \cite{Bellettini}. Any recovery sequence for this functional will be a recovery sequence for $\J_\eps$.
\end{proof}

\begin{proof}[Proof of Theorem \ref{thm:stefan}]
The first part of the proof is classical (See for instance \cite[Proposition 4.2]{M23} we recall the main steps here:
We note that $\rho^\eps$ solves \eqref{eq:NCH}
with
$$ \J_\eps[\rho^\eps(t)] + \int_0^t \int_{\R^d} \rho^\eps |v^\eps|^2\, dx\, dt \leq \J_\eps[\rho^\eps_{in}], \quad \sup_{t\in[0,T]}\int_{\R^d} |x|^2 \rho^\eps(t)\, dx \leq C(T). $$
We have in particular (see \cite{AGS08})
$$
d_W(\rho^\eps(t),\rho^\eps(s)) \leq \left(\int_0^T \int_{\R^d} \rho^\eps |v^\eps|^2\, dx\, dt \right)^{1/2} (t-s)^{1/2}
$$
so $\rho^\eps$ is bounded in $C^{1/2} ([0,T];\mathcal P_2(\R^d))$. Together with the bound on the second moment, this implies that  there exists a subsequence $\rho^{\eps_n}$ such that $\rho^{\eps_n}(t)$ narrowly converges to $\rho(t)$  locally uniformly in $t$ (see for example \cite[Proposition 4.1]{BE23}).
In addition, the fact that $m>2$ together with the bound on the energy implies that $\rho^\eps$ is bounded in $L^{\infty}(0,\infty;L^2(\R^d))$ and so $j^\eps =\rho^\eps v^\eps$ is bounded in $L^2(0,\infty;L^{4/3}(\R^d))$. Up to another subsequence, we can thus assume that $j^\eps$ converges to $j$ (weakly) and a classical argument implies that $j=\rho v$ for some $v\in L^2(d\rho(t))$ (see for example \cite[Proposition 4.1]{KMW2}).
We can pass to the limit in the continuity equation to show that $\rho,v$ satisfy the continuity equation 
$$ \pa_t \rho + \div (\rho v)=0.$$

To complete the proof, we will use the assumption of convergence of the energy \eqref{eq:EAS} to pass to the limit in the momentum equation.
First,   since $h(\rho^\eps)\geq h^{**}(\rho^\eps) \geq 0$ and 
$$
\liminf_{\eps\to 0} \int_0^T\int_{\R^d} h^{**}(\rho^\eps) \, dx \, dt \geq \int_0^T\int_{\R^d} h^{**}(\rho) \, dx \, dt
$$
(by the  convexity of $h^{**}$) and since $\frac{1}{4} \iint_{\R^d\times \R^d }  G_\eps(x-y)[\rho(x)-\rho(y)]^2\, dx\,dy \geq 0$, we see that \eqref{eq:EAS}
implies that
\begin{align}
& \int_0^T\int_{\R^d}  h(\rho^{\eps_n}) \vphi \, dx\, dt \to  \int_0^T\int_{\R^d}  h^{**}(\rho) \vphi \, dx\, dt \mbox{ for any  $\vphi\in L^\infty$},\label{eq:hv}\\
& h(\rho^{\eps_n})- h^{**}(\rho^{\eps_n})\to 0 \quad  \mbox{ in } L^1((0,T)\times\R^d),\label{eq:vgjh}\\
&  \int_0^T \iint_{\R^d\times \R^d }  G_{\eps_n}(x-y)[\rho^{\eps_n}(x)-\rho^{\eps_n}(y)]^2\, dx\,dy \, dt \to 0.\label{eq:ge0}
\end{align}
We now write the weak formulation of the momentum equation,  using the fact that $v^{\eps} = -\nabla h'(\rho^\eps) + \nabla B_\e(\rho^\e)$ (see \eqref{eq:NCH}):
\begin{align}
\int_0^T\int_{\R^d} \rho^{\eps_n} v^{\eps_n} \cdot \xi \, dx\,dt
& =  \int_0^T\int_{\R^d} \rho^{\eps_n} \na [-h'(\rho^{\eps_n})+B_\eps[\rho^{\eps_n}]] \cdot \xi \, dx\, dt \nonumber \\
& =  \int_0^T\int_{\R^d} [\rho^{\eps_n}  h'(\rho^{\eps_n})-h(\rho^{\eps_n})]\div \xi  \, dx\, dt -  \int_0^T\int_{\R^d} \rho^{\eps_n} \na B_{\eps_n}[\rho^{\eps_n}] \cdot \xi \, dx\, dt.\label{eq:mom}
\end{align}
We will  complete the proof of the theorem
by passing to the limit in \eqref{eq:mom}
 thanks to the following Lemma:
\begin{lemma}\label{eq:hjg}
Assume \eqref{eq:hv}, \eqref{eq:vgjh}, and \eqref{eq:ge0}. Then:
\item For all $\xi \in C^\infty_c((0,\infty)\times \R^d; \R^d)$, there holds
\begin{equation}\label{eq:lim1}
\lim_{n\to\infty} \int_0^T\int_{\R^d} \rho^{\eps_n} \na B_{\eps_n}[\rho^{\eps_n}] \cdot \xi \, dx\, dt = 0.
\end{equation}
\item For all $\xi \in C^\infty_c((0,\infty)\times \R^d; \R^d)$, there exists a subsequence $\eps_n'$ such that
\begin{equation}\label{eq:lim2}
\lim_{n\to\infty} 
\int_0^T\int_{\R^d} \left[\rho^{\eps_n'}  h'(\rho^{\eps_n'})-h(\rho^{\eps_n'})\right]\div \xi  \, dx\, dt
= \int_0^T\int_{\R^d}  \left[\rho   {h^{**}}'(\rho )-h^{**}(\rho ) \right] \div \xi  \, dx\, dt.
\end{equation}
\end{lemma}
This lemma implies (for a given $\xi \in C^\infty_c((0,\infty)\times \R^d; \R^d)$) and with the corresponding subsequence such that \eqref{eq:lim2} holds)
$$
\int_0^T\int_{\R^d} \rho v \cdot \xi \, dx\,dt
=\lim_{n\to\infty}\int_0^T\int_{\R^d} \rho^{\eps_n'} v^{\eps_n'} \cdot \xi \, dx\,dt
= \int_0^T\int_{\R^d}  \left[\rho   {h^{**}}'(\rho )-h^{**}(\rho ) \right] \div \xi  \, dx\, dt
$$
and since this holds for all  $\xi \in C^\infty_c((0,\infty)\times \R^d; \R^d)$, the theorem follows.
\end{proof}
It remains now to prove the above lemma.

\begin{proof}[Proof of Lemma \ref{eq:hjg}]
In order to prove \eqref{eq:lim1}, we use the same computation used in \cite{MW,KMW2} by writing
\begin{align}
\frac 1 2  \int_0^T  \iint_{\R^d\times \R^d }  G_\eps(x-y) & [\rho(x)-\rho(y)]^2\, dx\,dy\\
& = \int_0^T \iint_{\R^d\times \R^d }  \rho \left[\frac 1 \sigma \rho-\phi_\eps \right]\, dx\,dy\nonumber \\
& = \sigma \int_0^T \iint_{\R^d\times \R^d }   \left[\frac 1 \sigma \rho-\phi_\eps \right]^2 \, dx\,dy    +  \int_0^T \iint_{\R^d\times \R^d }  \phi ^\eps\left[\frac 1 \sigma \rho-\phi_\eps \right]\, dx\,dy\nonumber \\
& = \sigma \int_0^T \iint_{\R^d\times \R^d }   \left[\frac 1 \sigma \rho-\phi_\eps \right]^2 \, dx\,dy   +\frac{\eps^2 }{\sigma}  \int_0^T \iint_{\R^d\times \R^d } | \na \phi_\eps  |^2\, dx\,dy\nonumber
\end{align}
and 
\begin{align*}
 \int_0^T\!\! \!\int_{\R^d} \rho^\eps \na B_\eps[\rho^\eps] \cdot \xi \, dx\, dt
 & =  \int_0^T\!\! \!\int_{\R^d} \rho^\eps \na [\phi_\eps  -\sigma^{-1}\rho^\eps] \cdot \xi \, dx\, dt\\
 & =  \int_0^T\!\! \!\int_{\R^d} [ \rho^\eps -\sigma\phi_\eps ]\na [\phi_\eps  -\sigma^{-1}\rho^\eps] \cdot \xi \, dx\, dt+
  \int_0^T\!\! \!\int_{\R^d} \!\! \! -\sigma\phi_\eps \na [\phi_\eps  -\sigma^{-1}\rho^\eps] \cdot \xi \, dx\, dt\\
 & =\frac { \sigma }{2}\int_0^T\!\! \!\int_{\R^d}  [\phi_\eps  -\sigma^{-1}\rho^\eps]^2 \div \xi \, dx\, dt  +  \eps^2 \int_0^T\!\! \!\int_{\R^d} \phi_\eps  \na \Delta \phi_\eps    \cdot \xi \, dx\, dt\\
 & =\frac { \sigma }{2}\int_0^T\!\! \!\int_{\R^d}  [\phi_\eps  -\sigma^{-1}\rho^\eps]^2 \div \xi \, dx\, dt  \\
 & \quad+  \eps^2 \int_0^T\!\! \!\int_{\R^d} \frac{|\na \phi_\eps |^2}{2}\div\xi + D\xi :\na \phi_\eps \otimes\phi_\eps  + \phi_\eps  \na \phi_\eps \cdot  \Delta\xi  \, dx\, dt
\end{align*}
which together imply
$$
\left| \int_0^T\int_{\R^d} \rho^{\eps_n} \na B_{\eps_n}[\rho^{\eps_n}] \cdot \xi \, dx\, dt\right| \leq C(\xi)    \int_0^T \iint_{\R^d\times \R^d }  G_{\eps_n}(x-y)[\rho^{\eps_n}(x)-\rho^{\eps_n}(y)]^2\, dx\,dy.
$$
Thus \eqref{eq:ge0} yields \eqref{eq:lim1}.

\medskip

We now turn to the proof of \eqref{eq:lim2}.
Given $\xi \in C^\infty_c((0,\infty)\times \R^d; \R^d)$, we fix $Q=(0,T)\times B_R$ such that $\supp \xi\in Q$ and 
we introduce the measurable sets
$$E= \{ h^{**}(\rho)=0\}\cap Q , \qquad F= \{ h^{**}(\rho)>0\} \cap Q.  $$
First we claim that there exists a subsequence $\eps_n'$ such that
\begin{equation}\label{eq:claim1}
\lim_{n\to\infty} \iint_E \left[\rho^{\eps_n'}  h'(\rho^{\eps_n'})-h(\rho^{\eps_n'}) \right] \div \xi  \, dx\, dt \to 0.
\end{equation} 
Indeed, using \eqref{eq:hv}, we get $h(\rho^{\eps_n'}) \to 0 $ a.e. in $E$ (for some subsequence).
Furthermore, 
for any $\delta>0$, there exists $C_\delta$ such that $|\rho h'(\rho)-h(\rho) | \leq \delta + C_\delta h(\rho)$ (this is because the function in the left hand side vanishes when $\rho=0$ or $\theta$ which are the only zeroes of $h$ and the two functions have the same growth at $\infty$). 
We deduce
$$ \limsup_{n\to\infty } \iint_E \left|\rho^{\eps_n'}  h'(\rho^{\eps_n'})-h(\rho^{\eps_n'}) \right| |\div \xi | \, dx\, dt 
\leq \|\div \xi\|_{L^\infty}\limsup_{n\to\infty }   \iint_E  \delta + C_\delta h(\rho^{\eps_n'})  \, dx\, dt \leq  \|\div \xi\|_{L^\infty} |Q| \delta$$
and  \eqref{eq:claim1}  follows by taking $\delta\to 0$.
\medskip

Next, we will prove that the strict convexity if $h^{**}$ in $\{h^{**}(s)>0\}$ implies the strong convergence of $\rho^\eps$ to $\rho$ in $F$.
We define
$w^\eps = \max\{\rho^\eps,\theta\}$, which only takes values in the set where $\rho\mapsto h^{**}(\rho) $ is  strictly convex.
Since $w^{\eps_n}$ and $\rho^{\eps_n}$ are bounded in $L^2(Q)$, they both have some weak limit $w$ and $\rho$ (up to some  subsequence). And since $w^{\eps_n}\geq \rho^{\eps_n}$, we have $w\geq\rho$.
Furthermore, we have 
$$ h^{**}(w^{\eps_n}) = h^{**}(\rho^{\eps_n}) \leq h(\rho^{\eps_n})$$
(the first equality holds since $h^{**}(\rho)=0$ when $\rho<\theta$) and the convexity of $h^{**}$ implies
$$\int_Q h^{**}(w) \leq \liminf  \int_Q h^{**}(w^{\eps_n'})\leq  \limsup  \int_Q h^{**}(w^{\eps_n'})\leq  \limsup  \int_Q h(\rho^{\eps_n'}) = \int_Q h^{**}(\rho) \leq \int_Q h^{**}(w)$$
where we used the assumption of convergence of the energy \eqref{eq:EAS}. 
We deduce that we have equality in all these inequalities. In particular
$$ \lim  \int_Q h^{**}(w^{\eps_n'}) = \int_Q h^{**}(w) $$
and 
$$ \int_Q h^{**}(\rho) = \int_Q h^{**}(w).$$
The first equality and strict convexity of $h^{**}$ in the set of values of $w^{\eps_n'}(t,x)$ implies that 
$w^{\eps_n'}$ converges strongly (and a.e.) to $w$. The second equality (together with the fact that $w\geq\rho$ and the strict monotonicity of $h^{**}$ in $(\theta,\infty)$) implies that $\rho=w$ in $F$.
Finally, since $w^{\eps_n'}\geq \rho^{\eps_n'}$, we have 
$$\int_{ F} |w^{\eps_n'}-\rho^{\eps_n'}| = \int_{F} w^{\eps_n'}-\rho^{\eps_n'} \to \int_{F} w-\rho = 0$$
and so $\rho^{\eps_n'}$ converges to $w=\rho$ in $L^1(F)$ and a.e. (up to yet another subsequence).
Together with \eqref{eq:claim1}, this implies \eqref{eq:lim2}.
\end{proof}

\bibliographystyle{plain}
\bibliography{mybib}

\end{document}